\documentclass[aos,preprint]{imsart}
\setattribute{journal}{name}{}

\def \tcred{\textcolor{black}}
\usepackage[round,colon,authoryear]{natbib}
\RequirePackage[OT1]{fontenc}
\RequirePackage{amsthm,amsmath,bm,natbib,graphicx,enumitem}
\RequirePackage[colorlinks,citecolor=blue,urlcolor=blue]{hyperref}
\RequirePackage{hypernat}
\usepackage{amssymb,tabularx,multicol,multirow,booktabs}
\usepackage[usenames,dvipsnames]{color}
\RequirePackage{xr}

\newtheorem{theorem}{Theorem}

\newtheorem{lemma}{Lemma}

\newtheorem{lem}[lemma]{Lemma}

\DeclareMathOperator{\sech}{sech}

\def \bx{\mathbf{x}}
\def \by{\mathbf{y}}
\def \bQ{\mathbf{Q}}
\def \bX{\mathbf{X}}

\def \be{\begin{equs}}
\def \ee{\end{equs}}

\def \E{\mathbb{E}}
\def \P{\mathbb{P}}

\def \sumn{\sum_{i=1}^n}
\def \bmu{\pmb{\mu}}
\def \R{\mathbb{R}}

\def \Ptheta {\P_{\theta,\bmu}}
\def \Etheta {\E_{\theta,\bmu}}

\def \Pzero {\P_{\theta,\mathbf{0}}}
\def \Ezero {\E_{\theta,\mathbf{0}}}

\def \PJ {\P_{\mathbf{Q},\bmu}}

\def \EJ {\E_{\mathbf{Q},\bmu}}

\begin{document}

\begin{frontmatter}
	\title{Global Testing Against Sparse Alternatives under Ising Models}
	\runtitle{Detection Thresholds for Ising Models}
		\thankstext{m2}{The research of Sumit Mukherjee was supported in part by NSF Grant DMS-1712037.}
	\thankstext{m4}{The research of Ming Yuan was supported in part by NSF FRG Grant DMS-1265202, and NIH Grant 1-U54AI117924-01.}

	\begin{aug}
		
		\author{\fnms{Rajarshi} \snm{Mukherjee}\thanksref{m1}\ead[label=e2]{rmukherj@berkeley.edu}},
		\author{\fnms{Sumit} \snm{Mukherjee}\thanksref{m2}\ead[label=e1]{sm3949@columbia.edu}},
		\and
		\author{\fnms{Ming} \snm{Yuan}\thanksref{m4}\ead[label=e3]{ming.yuan@columbia.edu}}

		\affiliation{University of California, Berkeley\thanksmark{m1} and Columbia University \thanksmark{m2}\thanksmark{m4} }
		
		\address{Division of Biostatistics,\\
		Haviland Hall, Berkeley, CA- 94720. \\
			\printead{e2}}

		\address{Department of Statistics\\
			1255 Amsterdam Avenue\\
			New York, NY-10027. \\
			\printead{e1}\\
			\printead{e3}}
		

	\end{aug}

\begin{abstract} 
	In this paper, we study the effect of dependence on detecting sparse signals. In particular, we focus on global testing against sparse alternatives for the means of binary outcomes following an Ising model, and establish how the interplay between the strength and sparsity of a signal determines its detectability under various notions of dependence. The profound impact of dependence is best illustrated under the Curie-Weiss model where we observe the effect of a ``thermodynamic" phase transition. In particular, the critical state exhibits a subtle ``blessing of dependence'' phenomenon in that one can detect much weaker signals at criticality than otherwise. Furthermore, we develop a testing procedure that is broadly applicable to  account for dependence and show that it is asymptotically minimax optimal under fairly general regularity conditions.
\end{abstract}


\end{frontmatter}

%

\section{Introduction}

Motivated by applications in a multitude of scientific disciplines, statistical analysis of ``sparse signals" in a high dimensional setting, be it large-scale multiple testing or screening for relevant features, has drawn considerable attention in recent years. For more discussions on sparse signal detection type problems see, e.g., \cite{Jin1, arias2005near, arias2008searching, addario2010combinatorial, Jin2, Ingster5, cai2014rate, arias2015sparse, mukherjee2015hypothesis}, and references therein. A critical assumption often made in these studies is that the observations are independent. Recognizing the potential limitation of this assumption, several recent attempts have been made to understand the implications of dependence in both theory and methodology. See, e.g., \cite{hall2008properties, Jin2, Candes, wu2014detection, jin2014rare}. These earlier efforts, setting in the context of Gaussian sequence or regression models, show that it is important to account for dependence among observations, and under suitable conditions, doing so appropriately may lead to tests that are as powerful as if the observations were independent. However, it remains largely unknown how the dependence may affect our ability to detect sparse signals beyond Gaussian models. The main goal of the present work is to fill in this void. In particular, we investigate the effect of dependence on detection of sparse signals for Bernoulli sequences, a class of problems arising naturally in many genomics applications \citep[e.g.,][]{mukherjee2015hypothesis}.

Let $\bX=(X_1,\ldots,X_n)^\top\in \{\pm 1\}^n$ be a random vector such that $\P(X_i=+1)=p_i$. In a canonical multiple testing setup, we want to test collectively that $H_0: p_i=1/2$, $i=1,2,\ldots,n$. Of particular interest here is the setting when $X_i$'s may be dependent. A general framework to capture the dependence among a sequence of binary random variables is the so-called Ising models, which have been studied extensively in the literature \citep{ising1925beitrag, onsager1944crystal, Ellis_Newman, majewski2001ising, stauffer2008social, mezard2009information}. An Ising model specifies the joint distribution of $\bX$ as:
\begin{align}
\P_{\bQ,\bmu}(\bX=\bx):=\frac{1}{Z(\mathbf{Q}, \mathbf{\bmu})}\exp{\left(\frac{1}{2}\bx^\top\mathbf{Q} \bx+\bmu^\top\bx\right)},\qquad \forall \bx \in \{\pm 1\}^n,
\label{eqn:general_ising}
\end{align}
where $\mathbf{Q}$ is an $n \times n$ symmetric and hollow matrix, $\bmu:=(\mu_1,\ldots,\mu_n)^\top\in \mathbb{R}^n$, and $Z(\mathbf{Q}, \mathbf{\bmu})$ is a normalizing constant. Throughout the rest of the paper, the expectation operator corresponding to \eqref{eqn:general_ising} will be analogously denoted by $\EJ$. 
It is clear that the matrix $\mathbf{Q}$ characterizes the dependence among the coordinates of $\bX$, and $X_i$'s are independent if $\mathbf{Q}=\mathbf{0}$. Under model (\ref{eqn:general_ising}), the relevant null hypothesis can be expressed as $\bmu=\mathbf{0}$. More specifically, we are interested in testing it against a sparse alternative:
\begin{equation} 
H_0: \bmu=\mathbf{0} \quad {\rm vs} \quad H_1: \bmu \in \Xi(s,B), \label{eqn:hypo_sparse}
\end{equation}
 where
$${\Xi}(s,B):=\left\{\bmu\in \R^n:|{\rm supp}(\bmu)|= s,{\rm \ and\ } \min_{i\in {\rm supp}(\bmu)}\mu_i\geq B>0\right\},$$
and
$$
{\rm supp}(\bmu):=\{1\le i\le n:\mu_i\ne 0\}.
$$
Our goal here is to study the impact of $\mathbf{Q}$ in doing so.

To this end, we adopt an asymptotic minimax framework that can be traced back at least to \cite{burnashev1979minimax, ingster1994minimax, Ingster1}. See \cite{Ingster4} for further discussions. Let a statistical test for $H_0$ versus $H_1$ be a measurable $\{0,1\}$ valued function of the data $\bX$, with $1$ indicating rejecting the null hypothesis $H_0$ and $0$ otherwise. The worst case risk of a test $T: \{\pm 1\}^n\to \{0,1\}$ can be given by 
\begin{align} 
\mathrm{Risk}(T,{\Xi}(s,B),\bQ)&:=\P_{\bQ,\mathbf{0}}\left(T(\bX)=1\right)+\sup_{\bmu \in {\Xi}(s,B)}\P_{\bQ,\bmu}\left(T(\bX)=0\right),\label{eqn:general_hypo_ising}
\end{align}
where $\P_{\mathbf{Q},\bmu}$ denotes the probability measure as specified by \eqref{eqn:general_ising}. We say that a sequence of tests $T$ indexed by $n$ corresponding to a sequence of model-problem pair (\ref{eqn:general_ising}) and (\ref{eqn:general_hypo_ising}), to be asymptotically powerful (respectively asymptotically not powerful) against $\Xi(s,B)$ if
\begin{equation}
\label{eqn:powerful}
\limsup\limits_{n\rightarrow \infty}\mathrm{Risk}(T,{\Xi}(s,B),\bQ)= 0\text{ (respectively }\liminf\limits_{n\rightarrow \infty}\mathrm{Risk}(T,{\Xi}(s,B),\bQ)>0).
\end{equation}
The goal of the current paper is to characterize how the sparsity $s$ and strength $B$ of the signal $(\bmu)$ jointly determine if there is a powerful test, and how the behavior changes with $\bQ$. In particular,
\begin{enumerate}
\item[$\bullet$] for a general class of Ising models, we provide tests for detecting arbitrary sparse signals and show that they are asymptotically rate optimal for Ising models on regular graphs in the high temperature regime;
\item[$\bullet$] for Ising models on the cycle graph, we establish rate optimal results for all regimes of temperature, and show that the detection thresholds are the same as the independent case;  
\item[$\bullet$] for the Curie-Weiss model \citep{kac1969mathematical,nishimori2001statistical}, we provide sharp asymptotic detection thresholds for detecting arbitrarily sparse signals, which reveal an interesting phenomenon at the thermodynamic phase transition point of a Curie-Weiss magnet.
%
%
%
\end{enumerate} 
Our tools for analyzing the rate optimal tests depend on the method of exchangeable pairs \citep{chatterjee2007stein}, which might be of independent interest.

The rest of the paper is organized as follows. In Section \ref{section:curie_weiss} we study in detail the optimal detection thresholds for the Curie-Weiss model and explore the effects of the presence of a ``thermodynamic phase transition" in the model. Section \ref{section:main_results} is devoted to developing and analyzing testing procedures in the context of more general Ising models where we also show that under some conditions on $\bQ$, the proposed testing procedure is indeed asymptotically optimal. Finally we conclude with some discussions in Section \ref{section:discussion}. The proof of the main results is relegated to Section \ref{section:technical_details}. The proof of additional technical arguments can be found in \cite{mmy2017}.

\section{Sparse Testing under Curie-Weiss Model}\label{section:curie_weiss}

In most statistical problems, dependence reduces effective sample size and therefore makes inference harder. This, however,  turns out not necessarily to be the case in our setting. The effect of dependence on sparse testing under Ising model is more profound. To make this more clear we first consider one of the most popular examples of Ising models, namely the Curie-Weiss model. In the Curie-Weiss model,
\begin{align}\label{eq:Curie}
\Ptheta(\bX=\bx):=\frac{1}{Z(\theta,\bmu)}\exp\left(\frac{\theta}{n}\sum_{1\le i<j\le n}x_ix_j+\sum_{i=1}^n\mu_ix_i\right),
\end{align}
where in this section, with slight abuse of notation, we rename  $\P_{\bQ,\bmu},\E_{\bQ,\bmu},$ and $Z(\bQ,\bmu)$ by $\Ptheta$, $\Etheta$ and $Z(\theta,\bmu)$ respectively, for brevity.  The Curie-Weiss model is deceivingly simple and one of the classical examples that exhibit the so-called ``thermodynamic'' phase transition at $\theta=1$. See, e.g., \cite{kac1969mathematical,nishimori2001statistical}. It turns out that such a phase transition directly impacts how well a sparse signal can be detected. Following the convention, we shall refer to $\theta=1$ as the critical state, $\theta>1$ the low temperature states and $\theta<1$ the high temperature states.

\subsection{High temperature states}
We consider first the high temperature case i.e. $0\leq \theta<1$. It is instructive to begin with the case when $\theta=0$, that is, $X_1,\ldots, X_n$ are independent Bernoulli random variables. By Central Limit Theorem
$$
\sqrt{n}\left(\bar{X}-{1\over n}\sum_{i=1}^n \tanh(\mu_i)\right)\to_d N\left(0,{1\over n}\sum_{i=1}^n \sech^2(\mu_i)\right),
$$
where
$$\bar{X}={1\over n}\sum_{i=1}^n X_i.$$
In particular, under the null hypothesis,
$$
\sqrt{n}\bar{X}\to_d N\left(0,1\right).
$$
This immediately suggests a test that rejects $H_0$ if and only $\sqrt{n}\bar{X}\ge L_n$ for a diverging sequence $L_n=o(n^{-1/2}s\tanh(B))$ is asymptotic powerful, in the sense of (\ref{eqn:powerful}), for testing (\ref{eqn:hypo_sparse}) whenever $s\tanh(B)\gg n^{1/2}$. This turns out to be the best one can do in that there is no powerful test for testing (\ref{eqn:hypo_sparse}) if $s\tanh(B)=O(n^{1/2})$. See, e.g., \cite{mukherjee2015hypothesis}. An immediate question of interest is what happens if there is dependence, that is $0<\theta<1$. This is answered by Theorem \ref{thm:curie} below. 

\begin{theorem}\label{thm:curie}
Consider testing \eqref{eqn:hypo_sparse} based on $\bX$ following the Curie-Weiss model (\ref{eq:Curie}) with $0\leq \theta<1$. If $s\tanh(B)\gg n^{1/2}$, then the test that rejects $H_0$ if and only if $\sqrt{n}\bar{X}\ge L_n$ for a diverging $L_n$ such that $L_n=o(n^{-1/2}s\tanh(B))$ is asymptotically powerful for \eqref{eqn:hypo_sparse}. Conversely, if $s\tanh(B)=O( n^{1/2})$, then there is no asymptotically powerful test for \eqref{eqn:hypo_sparse}.
\end{theorem}

Theorem \ref{thm:curie} shows that, under high temperature states, the sparse testing problem \eqref{eqn:hypo_sparse} behaves similarly to the independent case. Not only the detection limit remains the same, but also it can be attained even if one neglects the dependence while constructing the test.

\subsection{Low temperature states}
Now consider the low temperature case when $\theta>1$. The na\"ive test that rejects $H_0$ whenever $\sqrt{n}\bar{X}\ge L_n$ is no longer asymptotically powerful in these situations. In particular, $\bar{X}$ concentrates around the roots of $x=\tanh(\theta x)$ and $\sqrt{n}\bar{X}$ is larger than any $L_n=O(n^{1/2})$ with a non-vanishing probability, which results in an asymptotically strictly positive probability of Type I error for a test based on rejecting  $H_0$ if $\sqrt{n}\bar{X}\ge L_n$.

To overcome this difficulty, we shall consider a slightly modified test statistic:
$$
\tilde{X}={1\over n}\sum_{i=1}^n \left[X_i-\tanh\left({\theta\over n}\sum_{j\neq i}X_j\right)\right],
$$
Note that
$$
\tanh\left({\theta\over n}\sum_{j\neq i}X_j\right)=\E_{\theta,\mathbf{0}}(X_i|X_j: j\neq i)
$$
is the conditional mean of $X_i$ given $\{X_j: j\neq i\}$ under the Curie-Weiss model with $\bmu=\mathbf{0}$. In other words, we average after centering each observation $X_i$ by its conditional mean, instead of the unconditional mean, under $H_0$. \tcred{The idea of centering by the conditional mean is similar in spirit to the pseudo-likelihood estimate of \cite{besag1974spatial,besag1975statistical}. See also \cite{guyon1995random,chatterjee2007estimation,bhattacharya2015inference}.}

We can then proceed to reject $H_0$ if and only if $\sqrt{n}\tilde{X}\ge L_n$. The next theorem shows that this procedure is indeed optimal with appropriate choice of $L_n$.

\begin{theorem}\label{thm:curie1}
Consider testing \eqref{eqn:hypo_sparse} based on $\bX$ following the Curie-Weiss model (\ref{eq:Curie}) with $\theta>1$. If $s\tanh(B)\gg n^{1/2}$, then the test that rejects $H_0$ if and only if $\sqrt{n}\tilde{X}\ge L_n$ for a diverging $L_n$ such that $L_n=o(n^{-1/2}s\tanh(B))$ is asymptotically powerful for \eqref{eqn:hypo_sparse}. Conversely, if $s\tanh(B)=O(n^{1/2})$, then there is no asymptotically powerful test for \eqref{eqn:hypo_sparse}.
\end{theorem}

Theorem \ref{thm:curie1} shows that the detection limits for low temperature states remain the same as that for high temperature states, but a different test is required to achieve it.

\subsection{Critical state}
The situation however changes at the critical state $\theta=1$, where a much weaker signal could still be detected. This is made precise by our next theorem, where we show that detection thresholds, in terms of $s\tanh(B)$, for the corresponding Curie-Weiss model at criticality scales as $n^{-3/4}$ instead of $n^{-1/2}$ as in either low or high temperature states. Moreover, it is attainable by the test that rejects $H_0$ whenever $n^{1/4}\bar{X}\ge L_n$ for appropriately chosen $L_n$.
\begin{theorem}\label{thm:corr}
Consider testing \eqref{eqn:hypo_sparse} based on $\bX$ following the Curie-Weiss model (\ref{eq:Curie}) with $\theta=1$. If $s\tanh(B)\gg n^{1/4}$, then a test that rejects $H_0$ if and only if $n^{1/4}\bar{X}\ge L_n$ for a suitably chosen diverging sequence $L_n$, is asymptotically powerful for \eqref{eqn:hypo_sparse}. Conversely, if $s\tanh(B)=O(n^{1/4})$, then there is no asymptotically powerful test for \eqref{eqn:hypo_sparse}.
\end{theorem}

A few comments are in order about the implications of Theorem \ref{thm:corr} in contrast to Theorem \ref{thm:curie} and \ref{thm:curie1}. Previously, the distributional limits for the total magnetization $\sum_{i=1}^n X_i$ has been characterized in all the three regimes of high $(\theta<1)$, low $(\theta>1)$, and critical $(\theta=1)$ temperatures \citep{Ellis_Newman} when $\bmu=\mathbf{0}$. \tcred{More specifically, they show that
\begin{eqnarray*}
\sqrt{n}\bar{X}&\stackrel{d}{\rightarrow}&N\Big(0,\frac{1}{1-\theta}\Big)\qquad\text{ if }\theta <1,\\
n^{1/4}\bar{X}&\stackrel{d}{\rightarrow}&W\qquad\text{ if }\theta=1,\\
\Big(\sqrt{n}(\bar{X}-m(\theta))|\bar{X}>0\Big)&\stackrel{d}{\rightarrow}&N\Big(0,\frac{1}{1-\theta(1-m(\theta)^2)}\Big)\qquad\text{ if }\theta>1,
\end{eqnarray*}
where $W$ is a random variable on $\R$ with density proportional to $e^{-x^4/12}$ with respect to Lebesgue measure, and $m(\theta)$ is the unique positive root of the equation $z=\tanh(\theta z)$ for $\theta>1$. A central quantity of their analysis is studying the roots of this equation.
}
 Our results demonstrate parallel behavior in terms of detection of sparse external magnetization $\bmu$. \tcred{In particular, if the vector $\bmu=(B,\cdots,B,0,0,\cdots,0)$ with the number of nonzero components equal to $s$, we obtain the fixed point equation $z=p\tanh(\theta z+B)+(1-p)\tanh(\theta z)$, where $p:={s}/{n}$. One can get an informal explanation of the detection boundary for the various cases from this fixed point equation. As for example in the critical case when $\theta=1$, we get the equation $$z=p\tanh(z+B)+(1-p)\tanh(z)\Rightarrow z-\tanh(z)=p[\tanh(z+B)-\tanh(z)].$$ The LHS of the second equality is of order $z^3$ for $z\approx 0$, and the RHS is of order $p\tanh(B)$. This gives the relation $z^3\sim (p\tanh B)$, which gives the asymptotic order of the mean of $\bar{X}$ under the alternative as $z\sim (p\tanh B)^{1/3}$. Since under $H_0$ the fluctuation of $\bar{X}$ is $n^{-1/4}$, for successful detection we need $n^{-1/4}\ll (p\tanh B)^{1/3}$, which is equivalent to $s\tanh(B)\gg n^{1/4}$ on recalling that $s=np$. Similar heuristic justification holds for other values of $\theta$ as well.}
 
 Interestingly, both below and above phase transition the detection problem considered here behaves similar to that in a disordered system of i.i.d. random variables, in spite having different asymptotic behavior of the total magnetization in the two regimes. However, an interesting phenomenon continues to emerge at $\theta=1$ where one can detect a much smaller signal or external magnetization (magnitude of $s\tanh(B)$). In particular, according to Theorem \ref{thm:curie} and Theorem \ref{thm:curie1}, no signal is detectable of sparsity $s\ll \sqrt{n}$, when $\theta \neq 1$. In contrast, Theorem \ref{thm:corr} establishes signals satisfying $s\tanh(B)\gg n^{1/4}$ is detectable for $n^{1/4}\lesssim s\ll \sqrt{n}$, where $a_n\lesssim b_n$ means $a_n=O(b_n)$. As mentioned before, it is well known the Curie-Weiss model undergoes a phase transition at $\theta=1$. Theorem \ref{thm:corr} provides a rigorous verification of the fact that the phase transition point $\theta=1$ can  reflect itself in terms of detection problems, even though $\theta$ is  a nuisance parameter. In particular, the detection is easier than at non-criticality. This is interesting in its own right since the concentration of $\bar{X}$ under the null hypothesis is weaker than that for $\theta<1$ \citep{chatterjee2010applications} and yet a smaller amount of signal enables us to break free of the null fluctuations. We shall make this phenomenon more transparent in the proof of the theorem.

\medskip

\section{Sparse Testing under General Ising Models}\label{section:main_results}
As we can see from the previous section, the effect of dependence on sparse testing under Ising models is more subtle than the Gaussian case. It is of interest to investigate to what extent the behavior we observed for the Curie-Weiss model applies to the more general Ising model, and whether there is a more broadly applicable strategy to deal with the general dependence structure. To this end, we further explore the idea of centering by the conditional mean we employed to treat low temperature states under Curie-Weiss model, and argue that it indeed works under fairly general situations.

\subsection{Conditional mean centered tests}
Note that under the Ising model (\ref{eqn:general_ising}),
$$
\E_{\bQ,\mathbf{0}}(X_i|X_j: j\neq i)=\tanh(m_i(\bX))
$$
where
$$
m_i(\bX)=\sum_{j=1}^n Q_{ij}X_j.
$$
Following the same idea as before, we shall consider a test statistic
$$
\tilde{X}={1\over n}\sum_{i=1}^n [X_i-\tanh(m_i(\bX))],
$$
and proceed to reject $H_0$ if and only if $\sqrt{n}\tilde{X}\ge L_n$. The following result shows that the same detection limit $s\tanh(B)\gg n^{1/2}$ can be achieved by this test as long as $\|\bQ\|_{\ell_\infty\to \ell_\infty}=O_p(1)$, where $\|\bQ\|_{\ell_p\to \ell_q} =\max_{\|x\|_{\ell_p}\le 1}\|\bQ x\|_{\ell_q}$ for $p,q>0$.

\begin{theorem}\label{thm:upper}
Let $\bX$ follow an Ising model \eqref{eqn:general_ising} with $\mathbf{Q}$ such that $\|\bQ\|_{\ell_\infty\to \ell_\infty}=O(1)$. Consider testing hypotheses about $\bmu$ as described by \eqref{eqn:hypo_sparse}. If $s\tanh(B)\gg n^{1/2}$, then the test that rejects $H_0$ if and only if $\sqrt{n}\tilde{X}\ge L_n$ for any $L_n\to\infty$ such that $L_n=o(n^{-1/2}s\tanh(B))$ is asymptotically powerful.
\end{theorem}

The condition $\|\bQ\|_{\ell_\infty\to \ell_\infty}=O(1)$ is a regularity condition which holds for many common examples of the Ising model in the literature. In particular, $\mathbf{Q}$ oftentimes can be associated with a certain graph $\mathcal{G}=(V,E)$ with vertex set $V=[n]:=\{1,\ldots,n\}$ and edge set $E \subseteq [n]\times [n]$ so that $\mathbf{Q}=(n\theta) G/(2|E|)$, where $G$ is the adjacency matrix for $\mathcal{G}$, $|E|$ is the cardinality of $E$, and $\theta \in \mathbb{R}$ is a parameter independent of $n$ deciding the degree of dependence in the spin-system. Below we provide several more specific examples that are commonly studied in the literature.
\paragraph{Dense Graphs:}  Recall that
\begin{align*}
\|\bQ\|_{\ell_\infty\to\ell_\infty}=\max_{1\le i\le n}\sum_{j=1}^n |Q_{ij}|\leq \frac{n^2|\theta|}{2|E|}.
\end{align*}
If the dependence structure is guided by densely labeled graphs so that $|E|=\Theta(n^2)$, then $\|\bQ\|_{\ell_\infty\to \ell_\infty}=O(1)$.
\paragraph{Regular Graphs:} When the dependence structure is guided by a regular graph of degree $d_n$,  we can write $\mathbf{Q}=\theta G/d_n$. Therefore,
\begin{align*}
\|\bQ\|_{\ell_\infty\to\ell_\infty}=\max_{1\le i\le n}\sum_{j=1}^n |Q_{ij}|=\frac{|\theta|}{d_n}\cdot d_n=|\theta|,
\end{align*}
and again obeying the condition $\|\bQ\|_{\ell_\infty\to \ell_\infty}=O(1)$.
\paragraph{Erd\"{o}s-R\'{e}nyi Graphs:} Another example is the Erd\"{o}s-R\'{e}nyi graph where an edge between each pair of nodes is present with probability $p_n$ independent of each other. It is not hard to derive from Chernoff bound and union bounds that the maximum degree $d_{\max}$ and the totally number of edges $|E|$ of an Erd\"{o}s-R\'{e}nyi graph satisfy with high probability:
$$
d_{\max}\le np_n(1+\delta),\qquad {\rm and}\qquad |E|\ge {n(n-1)\over 2}p_n\cdot(1-\delta)
$$
for any $\delta\in (0,1)$, provided that $np_n\gg \log{n}$. This immediately implies that $\|\bQ\|_{\ell_\infty\to\ell_\infty}=O_p(1)$.
\medskip

In other words, the detection limit established in Theorem \ref{thm:upper} applies to all these types of Ising models. In particular, it suggests that, under Curie-Weiss model, the $\sqrt{n}\tilde{X}$ based test can detect sparse external magnetization $\bmu\in \Xi(s,B)$ if $s\tanh(B)\gg n^{1/2}$, for any $\theta\in {\mathbb R}$, which, in the light of Theorems \ref{thm:curie} and \ref{thm:curie1}, is optimal in both high and low temperature states.

\subsection{Optimality}\label{section:lower_bound}
The detection limit presented in Theorem \ref{thm:upper} matches those obtained for independent Bernoulli sequence model. It is of interest to understand to what extent the upper bounds in Theorem \ref{thm:upper} are sharp. The answer to this question might be subtle. In particular, as we see in the Curie-Weiss case, the optimal rates of detection thresholds depend on the presence of thermodynamic phase transition in the null model. To further illustrate the role of criticality, we now consider an example of the Ising model without phase transition and the corresponding behavior of the detection problem \eqref{eqn:hypo_sparse} in that case. Let
$$\mathbf{Q}_{i,j}=\frac{\theta}{2} {\mathbb I}\{|i-j|=1\mod n\}$$
so that the corresponding Ising model can be identified with a cycle graph of length $n$. Our next result shows that the detection threshold remains the same for any $\theta$, and is the same as the independent case i.e. $\theta=0$.

\begin{theorem}\label{thm:z1_unstructured}
	Suppose $\bX\sim \P_{\mathbf{Q},\bmu}$, where $\mathbf{Q}$ is the scaled adjancency matrix of the cycle graph of length $n$, that is, $\mathbf{Q}_{i,j}=\frac{\theta}{2} 1\{|i-j|=1\mod n\}$ for some $\theta \in \mathbb{R}$. If $s\tanh(B)\le C\sqrt{n}$ for some $C>0$, then no test is asymptotically powerful for the testing problem (\ref{eqn:hypo_sparse}).
\end{theorem}

In view of Theorem \ref{thm:upper},  if $s\tanh(B)\gg n^{1/2}$, then the test that rejects $H_0$ if and only if $\sqrt{n}\tilde{X}\ge L_n$ for any $L_n\rightarrow\infty$ such that $L_n=o(n^{-1/2}s\tanh(B))$ is asymptotically powerful for the testing problem (\ref{eqn:hypo_sparse}). Together with Theorem \ref{thm:z1_unstructured}, this shows that for the Ising model on the cycle graph of length $n$, which is a physical model without thermodynamic phase transitions, the detection thresholds mirror those obtained in independent Bernoulli sequence problems \citep{mukherjee2015hypothesis}. 

The difference between these results and those for the Curie-Weiss model demonstrates the difficulty of a unified and complete treatment to general Ising models. We offer here, instead, a partial answer and show that the test described earlier in the section (Theorem \ref{thm:upper}) is indeed optimal under fairly general weak dependence for reasonably regular graphs. 

\begin{theorem}\label{thm:lower}
Suppose $\bX \sim \P_{\mathbf{Q},\bmu}$ as in \eqref{eqn:general_ising} and consider testing hypotheses about $\bmu$ as described by \eqref{eqn:hypo_sparse}. Assume $\mathbf{Q}_{i,j}\geq 0$ for all $(i,j)$ such that $\|\bQ\|_{\ell_\infty\to\ell_\infty}\le\rho<1$ for some constant $\rho>0$, $\|\bQ\|_{\rm F}^2=O(\sqrt{n})$, and
$$
\left\|\bQ{\mathbf 1}-\frac{{\mathbf 1}^\top \bQ{\mathbf 1}}{n}{\bf 1}\right\|^2=O(1).
$$
If $s\tanh(B)\le C\sqrt{n}$ for some constant $C>0$, then no test is asymptotically powerful for \eqref{eqn:hypo_sparse}.
\end{theorem}

Theorem \ref{thm:lower} provides rate optimal lower bound to certain instances pertaining to Theorem \ref{thm:upper}. One essential feature of Theorem \ref{thm:lower} is the implied impossibility result for the $s \ll \sqrt{n}$ regime. More precisely, irrespective of signal strength, no tests are asymptotically powerful when the number of signals drop below $\sqrt{n}$ in asymptotic order. This is once again in parallel to results in \cite{mukherjee2015hypothesis}, and provides further evidence that low dependence/high temperature regimes $(\text{as encoded by }\ \|\bQ\|_{\ell_\infty\to\ell_\infty}\le\rho<1)$ resemble independent Bernoulli ensembles. Theorem \ref{thm:lower} immediately implies the optimality of the conditional mean centered tests for a couple of common examples.

\paragraph{High Degree Regular Graphs:} When the dependence structure is guided by a regular graph, that is $\mathbf{Q}=\frac{\theta}{d_n}G$, it is clear that
$$
\left\|\bQ{\mathbf 1}-\frac{{\mathbf 1}^\top \bQ{\mathbf 1}}{n}{\bf 1}\right\|^2=0.
$$
If $0\leq \theta<1$ and $d_n\gtrsim \sqrt{n}$, then one can easily verify the conditions of Theorem \ref{thm:lower} since
$$\|\mathbf{Q}\|_{\ell_\infty\to \ell_\infty}=\theta<1,\qquad{\rm and}\qquad \|\bQ\|_{\rm F}^2=n\theta^2/d_n.$$
\paragraph{Dense Erd\"{o}s-R\'{e}nyi Graphs:}
When the dependence structure is guided by a Erd\"{o}s-R\'{e}nyi graph on $n$ vertices with parameter $p_n$, that is $\mathbf{Q}=\theta/(np_n)G$ with $G_{i,j}\sim \mathrm{Bernoulli}(p_n)$ independently for all $1\le i<j\le n$, we can also verify that the conditions of Theorem \ref{thm:lower} holds with probability tending to one if $0\leq \theta<1$ and $p_n$ bounded away from $0$. As before, by Chernoff bounds, we can easily derive that with probability tending to one,
$$
\|\bQ\|_{\ell_\infty\to\ell_\infty}=\theta\frac{d_{\max}}{np_n}\leq \frac{\theta(1+\delta)np_n}{np_n}=\theta(1+\delta),
$$
and
$$
\|\bQ\|_{\rm F}^2={\theta^2\over n^2p_n^2}\sum_{1\le i<j\le n}G_{i,j}\le {\theta^2(1+\delta)n(n-1)p_n\over 2n^2p_n^2}\le {\theta^2\over 2p_n}(1+\delta),
$$
for any $\delta>0$. Finally, denote by $d_i$ the degree of the $i$th node, then
$$
\left\|\bQ{\mathbf 1}-\frac{{\mathbf 1}^\top \bQ{\mathbf 1}}{n}{\bf 1}\right\|^2=\frac{\theta^2}{n^2p_n^2}\sum_{i=1}^n\left(d_i-\frac{1}{n}\sum_{j=1}^n d_j\right)^2\leq \frac{\theta^2}{n^2p_n^2}\sum_{i=1}^n\left(d_i-(n-1)p_n\right)^2=O_p(1),
$$
by Markov inequality and the fact that
$$
\E\left[\sum_{i=1}^n\left(d_i-(n-1)p_n\right)^2\right]=n(n-1)p_n(1-p_n).$$
\medskip

\section{Simulation Results}
\label{section:simulation}

We now present results from a set of numerical experiments to further demonstrate the behavior of the various tests in finite samples. To fix ideas, we shall focus on the Curie-Weiss model since it exhibits the most interesting behavior in terms of the effect of thermodynamic phase transitions reflecting itself on the detection thresholds for the presence of sparse magnetization. In order to demonstrate the detection thresholds cleanly in the simulation, we parametrized sparsity $s$ as $s=n^{1-\alpha}$ for $\alpha\in (0,1)$. In this parametrization, the theoretical detection thresholds obtained for the Curie-Weiss model can be restated as follows. For $\theta \neq 1$, Theorem \ref{thm:curie} and Theorem \ref{thm:curie1} suggest that the critical signal strength equals $\tanh(B)\sim n^{-(\frac{1}{2}-\alpha)}$. In particular if $\tanh(B)=n^{-r}$, then no test is asymptotically powerful when $r>\frac{1}{2}-\alpha$; whereas the test based on conditionally centered magnetization is asymptotically powerful when $r<\frac{1}{2}-\alpha$. Moreover, for $\alpha>{1}/{2}$, all tests are asymptotically powerless irrespective of the amount of signal strength.  
However $\theta =1$, Theorem \ref{thm:corr} demonstrates that the critical signal strength equals $\tanh(B)\sim n^{-(\frac{3}{4}-\alpha)}$. In particular if $\tanh(B)=n^{-r}$, then no test is asymptotically powerful when $r>\frac{3}{4}-\alpha$; whereas the test based on total  magnetization is asymptotically powerful when $r<\frac{3}{4}-\alpha$. Moreover, for $\alpha>{3}/{4}$, all tests are asymptotically powerless irrespective of the amount of the signal strength. The simulation presented below is designed to capture the different scenarios where non-trivial detection is possible i.e. $\alpha\leq {1}/{2}$ for $\theta\neq 1$ and $\alpha\leq {3}/{4}$ for $\theta=1$.

We evaluated the power of the two tests, based on total magnetization and the conditionally centered magnetization respectively, at the significance level of $5\%$ and sample size $n=1000$. We generated the test statistics $500$ times under the null and take the $95\%$-quantile as the critical value. The power against different alternatives are then obtained empirically from $500$ repeats each. The simulation from a Curie-Weiss model in the presence of magnetization is done using the Gaussian trick or the auxiliary variable approach as demonstrated by Lemma \ref{lem:bayesian}. In particular for a given $\theta$ and $\bmu$ in the simulation parameter set, we generated a random variable $Z$ (using package \verb+rstan+ in \verb+R+) with density proportional to $f_{n,\bmu}(z):=\frac{n\theta z^2}{2}-\sum_{i=1}^n\log\cosh(\theta z+\mu_i)$. Next, given this realization of $Z=z$ we generated each component of $\bX=(X_1,\ldots,X_n)$ independently taking values in $\pm 1$ with 
$\Ptheta(X_i=x_i)=\frac{e^{(\mu_i+z\theta)x_i}}{e^{\mu_i+z\theta}+e^{-\mu_i-z\theta}}$. Thereafter Lemma \ref{lem:bayesian} guarantees the joint distribution of $\bX$ indeed follows a Curie-Weiss model with temperature parameter $\theta$ and magnetization $\bmu$. We believe that this method is much faster than the one-spin at a time Glauber dynamics which updates the whole chain $\bX$ one location at a time. We have absorbed all issues regarding mixing time in the simulation of $Z$, which being a one dimensional continuous random variable behaves much better in simulation.

In figure \ref{fig:dense_signal}, we plot the power of both tests for $\theta=0.5$ (high temperature, conditionally centered magnetization), $\theta=1$ (critical temperature, total magnetization), and $\theta=1.5$ (low temperature, conditionally centered magnetization). Each plot was produced by repeating the experiment for a range of equally spaced signal sparsity-strength pairs $(\alpha,r)$ with an increment of size $0.05$. In addition, we plot in red the theoretical detection boundary given by $r={1}/{2}-\alpha$ for non-critical temperature ($\theta\neq 1$) and $r={3}/{4}-\alpha$ for critical temperature ($\theta=1$). These simulation results agree very well with our theoretical development. 


\begin{figure}
	\begin{center}
		\includegraphics[width=7 cm,keepaspectratio]{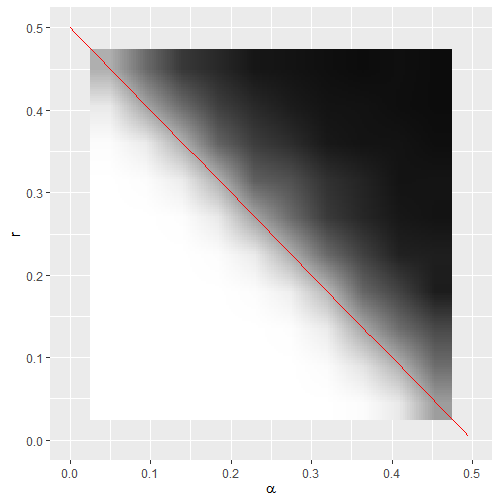}
		\includegraphics[width=7 cm,keepaspectratio]{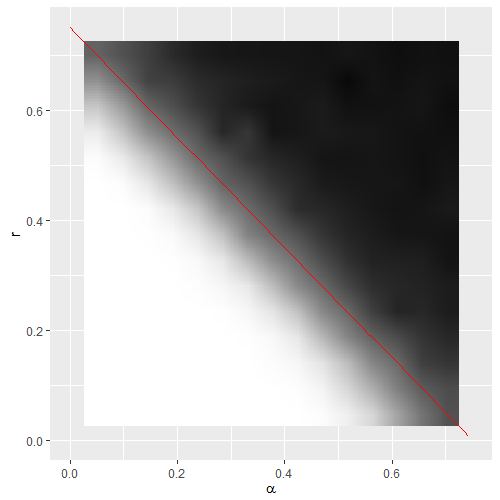}
		\includegraphics[width=7 cm,keepaspectratio]{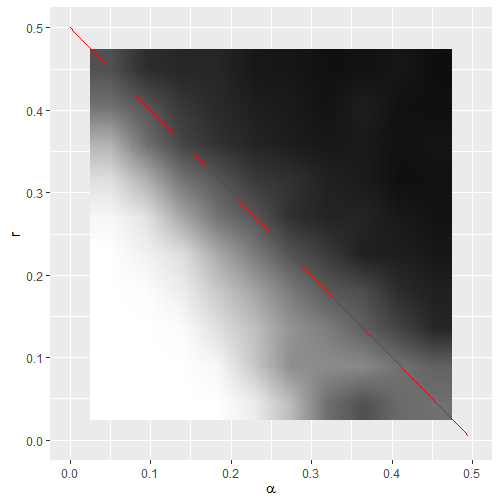}
		\caption{The power of testing procedures in the dense signal setup. (a) shows the power of the conditionally centered magnetization test for $\theta=0.5$, (b) shows the power of the total magnetization test for $\theta=1$ (c) shows the power of the conditionally centered magnetization test for $\theta=1.5$ \label{fig:dense_signal}. The theoretical detection threshold is drawn in red.}
	\end{center}
\end{figure}


\section{Discussions}\label{section:discussion}
In this paper we study the asymptotic minimax rates of detection for arbitrary sparse signals in Ising Models, considered as a framework to study dependency structures in binary outcomes. We show that the detection
thresholds in Ising models might depend on the presence of a ``thermodynamic" phase transition in the model. In the context of a Curie-
Weiss Ising model, the presence of such a phase transition results in
substantial faster rates of detection of sparse signals at criticality. On
the other hand, lack of such phase transitions, in the Ising model on
the line graph, yields results parallel to those in independent Bernoulli
sequence models, irrespective of the level of dependence. We further show that for Ising models defined on graphs enjoying certain degree of regularity, detection thresholds parallel those in independent Bernoulli
sequence models in the low dependence/high temperature regime. It will be highly interesting to consider other kinds of graphs left out by Theorem \ref{thm:lower} in the context of proving matching lower bounds to Theorem \ref{thm:upper}. This seems highly challenging and might depend heavily on the sharp asymptotic behavior of the partition function of more general Ising model under low-magnetization regimes. The issue of unknown dependency structure $\bQ$, and especially the estimation of unknown temperature parameter $\theta$ for Ising models defined on given underlying graphs, is also subtle  as shown in \cite{bhattacharya2015inference}. In particular, the rate of consistency of an estimator of $\theta$ under the null model (i.e. $\bmu=0$) depends crucially on the position of $\theta$ with respect to the point of criticality and in particular high temperature regimes (i.e. low positive values of $\theta$) may preclude the existence of any consistent estimator. The situation becomes even more complicated in presence of external magnetization (i.e. $\bmu\neq 0$). Finally, this paper opens up several
interesting avenues of future research. In particular, investigating the effect of dependence on detection of segment type structured signals deserves special attention.

\section{Proof of Main Results}\label{section:technical_details}
In this section we collect the proofs of our main results. It is convenient to first prove the general results, namely the upper bound given by Theorem \ref{thm:upper} and lower bound by Theorem \ref{thm:lower}, and then consider the special cases of the Ising model on a cycle graph, and Curie-Weiss model.

\subsection{Proof of Theorem \ref{thm:upper}}

The key to the proof is the tail behavior of
$$
f_{\bQ,\bmu}(\bX):=\frac{1}{n}\sum_{i=1}^n \left[X_i-\E_{\bQ,\bmu}(X_i|X_j: j\neq i)\right]=\frac{1}{n}\sum_{i=1}^n \left[X_i-\tanh\left(\sum_{j\neq i}\bQ_{ij}X_j+\mu_j\right)\right],
$$
where $\E_{\bQ,\bmu}$ means the expectation is taken with respect to the Ising model (\ref{eqn:general_ising}). In particular, we shall make use of the following concentration bound for $f_{\bQ,\bmu}(\bX)$.

\begin{lemma}\label{lemma_chatterjee}
Let $\bX$ be a random vector following the Ising model \eqref{eqn:general_ising}. Then for any $t>0$,
$$
\PJ( |f_{\bQ,\bmu}(\bX)|\ge t)\le  2\exp\left\{-\frac{nt^2}{4\left(1+\|\bQ\|_{\ell_\infty\to \ell_\infty}\right)^2}\right\}.
$$
\end{lemma}

Lemma \ref{lemma_chatterjee} follows from a standard application of Stein's Method for concentration inequalities \citep{chatterjee2005concentration, chatterjee2007stein, chatterjee2010applications}. We defer the detailed proof to the Appendix.

We are now in position to prove Theorem \ref{thm:upper}. We first consider the Type I error. By Lemma \ref{lemma_chatterjee}, there exists a constant $C>0$ such that
$$
\P_{\bQ,\mathbf{0}}(\sqrt{n}\tilde{X}\ge L_n)\le 2\exp(-CL_n^2)\to 0.
$$
It remains to consider the Type II error. Note that
\begin{eqnarray*}
\tilde{X}-f_{\bQ,\bmu}(\bX)&=&{1\over n}\sum_{i=1}^n\left[\tanh\left(\sum_{j\neq i}\bQ_{ij}X_j+\mu_i\right)-\tanh\left(\sum_{j\neq i}\bQ_{ij}X_j\right)\right]\\
&=&{1\over n}\sum_{i\in {\rm supp}(\bmu)}\left[\tanh\left(\sum_{j\neq i}\bQ_{ij}X_j+\mu_i\right)-\tanh\left(\sum_{j\neq i}\bQ_{ij}X_j\right)\right]\\
&\ge&{1\over n}\sum_{i\in {\rm supp}(\bmu)}\left[\tanh\left(\sum_{j\neq i}\bQ_{ij}X_j+B\right)-\tanh\left(\sum_{j\neq i}\bQ_{ij}X_j\right)\right],
\end{eqnarray*}
where the inequality follows from the monotonicity of $\tanh$.

Observe that for any $x\in \R$ and $y>0$,
\begin{equation}
\label{eq:lowertanh}
\tanh(x+y)-\tanh(x)={[1-\tanh^2(x)]\tanh(y)\over 1+\tanh(x)\tanh(y)}\ge [1-\tanh(x)]\tanh(y),
\end{equation}
where the inequality follows from the fact that $|\tanh(x)|\le 1$. Thus,
$$
\tilde{X}-f_{\bQ,\bmu}(\bX)\ge {\tanh(B)\over n}\sum_{i\in {\rm supp}(\bmu)}\left[1-\tanh\left(\sum_{j\neq i}\bQ_{ij}X_j\right)\right].
$$
Because
$$
\sum_{j\neq i}\bQ_{ij}X_j\le \|\bQ\|_{\ell_\infty\to\ell_\infty},
$$
we get
$$
\tilde{X}-f_{\bQ,\bmu}(\bX)\ge {s\tanh(B)\over n}\left[1-\tanh\left(\|\bQ\|_{\ell_\infty\to\ell_\infty}\right)\right].
$$
Therefore,
$$
\sqrt{n}\tilde{X}-\sqrt{n}f_{\bQ,\bmu}(\bX)\ge {s\tanh(B)\over \sqrt{n}}\left[1-\tanh\left(\|\bQ\|_{\ell_\infty\to\ell_\infty}\right)\right]\gg L_n.
$$
This, together with another application of Lemma \ref{lemma_chatterjee}, yields the desired claim.

\subsection{Proof of Theorem \ref{thm:lower}}
The proof is somewhat lengthy and we break it into several steps.

\subsubsection{Reduction to magnetization}
\label{likelihood_ratio}
We first show that a lower bound can be characterizing the behavior of $\bar{X}$ under the alternative. To this end, note that for any test $T$ and a distribution $\pi$ over $\Xi(s,B)$, we have
\begin{eqnarray*}
\mathrm{Risk}(T,\Xi_{s,B},\bQ)&=&\P_{\bQ,\mathbf{0}}\left(T(\bX)=1\right)+\sup_{\bmu \in {\Xi}(s,B)}\P_{\bQ,\bmu}\left(T(\bX)=0\right)\\
&\geq& \P_{\bQ,\mathbf{0}}\left(T(\bX)=1\right)+\int\P_{\bQ,\bmu}\left(T(\bX)=0\right)d\pi(\bmu).
\end{eqnarray*}
The rightmost hand side is exactly the risk when testing $H_0$ against a simple alternative where $\bX$ follows a mixture distribution:
$$
\P_\pi(\bX=\bx) := \int\P_{\bQ,\bmu}\left(\bX=\bx\right)d\pi(\bmu)
$$
By Neymann-Pearson Lemma, this can be further lower bounded by
$$
\mathrm{Risk}(T,\Xi(s,B),\bQ)\ge \P_{\bQ,\mathbf{0}}\left(L_{\pi}(\bX)>1\right)+\int\P_{\bQ,\bmu}\left(L_{\pi}(\bX)\leq 1\right)d\pi(\bmu),
$$
where
$$
L_\pi(\bX)={\P_\pi(\bX)\over \P_{\bQ,\mathbf{0}}(\bX)}
$$
is the likelihood ratio.

We can now choose a particular prior distribution $\pi$ to make $L_\pi$ a monotone function of $\bar{X}$. To this end, let $\pi$ be supported over
$$
\widetilde{\Xi}(s,B)=\left\{\bmu\in \{0,B\}^n: |{\rm supp}(\bmu)|=s\right\},
$$
so that
$$
\pi(\bmu)\propto Z(\bQ,\bmu),\qquad \forall \bmu\in \widetilde{\Xi}(s,B).
$$
It is not hard to derive that, with this particular choice,
$$
L_{\pi}(\bX)\propto \sum_{\bmu\in \widetilde{\Xi}(s,B)} \exp(\bmu^\top \bX)=\E_S \exp\left(B \sum_{i\in S} X_i\right),
$$
where $\E_S$ means expectation over $S$, a uniformly sampled subset of $[n]$ of size $s$. It is clear, by symmetry, that the rightmost hand side is invariant to the permutation of the coordinates of $\bX$. In addition, it is an increasing function of
$$|\{i\in [n]: X_i=1\}|={1\over 2}\left(n+\sum_{i=1}^n X_i\right),$$
and hence an increasing function of $\bar{X}$.

The observation that $L_\pi(\bX)$ is an increasing function of $\bar{X}$ implies that there exists a sequence $\kappa_n$ such that
\begin{eqnarray*}
\mathrm{Risk}(T,\Xi(s,B),\bQ)&\ge&\P_{\bQ,\mathbf{0}}\left(L_{\pi}(\bX)>1\right)+\int\P_{\bQ,\bmu}\left(L_{\pi}(\bX)\leq 1\right)d\pi(\bmu)\\
&=&\P_{\bQ,\mathbf{0}}\left(\sum_{i=1}^n X_i>\kappa_n\right)+\int\P_{\bQ,\bmu}\left(\sum_{i=1}^n X_i\leq \kappa_n\right)d\pi(\bmu)\\
&\ge&\P_{\bQ,\mathbf{0}}\left(\sum_{i=1}^n X_i>\kappa_n\right)+\inf_{\bmu \in \widetilde{\Xi}(s,B)}\P_{\bQ,\bmu}\left(\sum_{i=1}^n X_i\leq \kappa_n\right).
\end{eqnarray*}
It now remains to study the behavior of $\bar{X}$.

In particular, it suffices to show that, for any fixed $x>0$,
\begin{equation}
\label{eq:claim_lower}
\liminf_{n\to\infty}\P_{\bQ,\mathbf{0}}\left\{\sum_{i=1}^n X_i>x\sqrt{n}\right\}>0,
\end{equation}
and for any $x_n\to \infty$,
\begin{equation}
\label{eq:claim_upper}
\limsup_{n\to\infty}\sup_{\bmu \in \widetilde{\Xi}(s,B)}\P_{\bQ,\bmu}\left(\sum_{i=1}^n X_i> x_n\sqrt{n}\right)=0.
\end{equation}
Assuming \eqref{eq:claim_lower} holds, then for any test $T$ to be asymptotic powerful, we need $\kappa_n\gg \sqrt{n}$ to ensure that
$$
\P_{\bQ,\mathbf{0}}\left\{\sum_{i=1}^n X_i>\kappa_n\right\}\to 0.
$$
But, in the light of \eqref{eq:claim_upper}, this choice necessarily leads to
$$
\inf_{\bmu \in \widetilde{\Xi}(s,B)}\P_{\bQ,\bmu}\left\{\sum_{i=1}^n X_i\leq \kappa_n\right\}\to 1,
$$
so that
$$
\mathrm{Risk}(T,\Xi(s,B),\bQ)\to 1.
$$
In other words, there is no asymptotic powerful test if both \eqref{eq:claim_lower} and \eqref{eq:claim_upper} hold. We now proceed to prove them separately.

\subsubsection{Proof of \eqref{eq:claim_upper}:} Recall that $m_i(\bX)=\sum_{j=1}^n \bQ_{ij}X_j$ and assume $\bmu\in \tilde{\Xi}(s,B)$ with $s\tanh(B)\le C\sqrt{n}$. Also let $\mathbf{r}=(r_1,\ldots,r_n)^\top$ where $\mathbf{r}=\mathbf{r}(\bQ):=\bQ {\bf 1}$. We  split the proof into two cases, depending on whether $B\le 1$ or $B>1$.
	
\paragraph{The case of $B\in [0,1]:$} Write
\begin{eqnarray*}
\sum_{i=1}^n X_i&=&\sum_{i=1}^n \left[X_i-\tanh(m_i(\bX)+\mu_i)\right]+\sum_{i=1}^n \left[\tanh(m_i(\bX)+\mu_i)-\tanh(m_i(\bX))\right]\\
&&+\sum_{i=1}^n \left[\tanh(m_i(\bX))-m_i(\bX)\right]+\sum_{i=1}^n m_i(\bX).
\end{eqnarray*}
Observe that, 
$$
\sum_{i=1}^n m_i(\bX)={\bf 1}^\top \bQ \bX=\sum_{i=1}^n r_iX_i =\rho^* \sum_{i=1}^n X_i+\sum_{i=1}^n (r_i-\rho_\ast)X_i,
$$
where $\rho_\ast=\frac{1}{n}\mathbf{1}^\top\mathbf{r}=\frac{1}{n}\mathbf{1}^\top \bQ \mathbf{1}$. Thus,
\begin{eqnarray*}
(1-\rho_\ast)\sum_{i=1}^n X_i&=&\sum_{i=1}^n \left[X_i-\tanh(m_i(\bX)+\mu_i)\right]+\sum_{i=1}^n \left[\tanh(m_i(\bX)+\mu_i)-\tanh(m_i(\bX))\right]\\
&&+\sum_{i=1}^n \left[\tanh(m_i(\bX))-m_i(\bX)\right]+\sum_{i=1}^n (r_i-\rho_\ast)X_i.\\
&=:& \Delta_1 +\Delta_2+\Delta_3+\Delta_4.
\end{eqnarray*}

It is clear that
$$
\P_{\bQ,\bmu}\left\{\sum_{i=1}^n X_i> x_n\sqrt{n}\right\}\le \sum_{j=1}^4\P_{\bQ,\bmu}\left\{\Delta_j> {1\over 4(1-\rho_\ast)}x_n\sqrt{n}\right\}.
$$
We now argue that for any $x_n\to \infty$,
\begin{equation}
\label{eq:bdDelta}
\sup_{\bmu \in \widetilde{\Xi}(s,B)}\P_{\bQ,\bmu}\left\{\Delta_j> {1\over 4(1-\rho_\ast)}x_n\sqrt{n}\right\}\to 0,\qquad j=1,\ldots,4.
\end{equation}
The case for $\Delta_4$ follows from our assumption $(\left\|\bQ{\mathbf 1}-\frac{{\mathbf 1}^\top \bQ{\mathbf 1}}{n}{\bf 1}\right\|^2=O(1))$ upon Cauchy-Schwarz inequality. The case $\Delta_1$ follows immediately from Lemma \ref{lemma_chatterjee}. On the other hand, we note that
\begin{eqnarray*}
\sum_{i=1}^n \left[\tanh(m_i(\bX)+\mu_i)-\tanh(m_i(\bX))\right]&\le& \sum_{i=1}^n \left|\tanh(m_i(\bX)+\mu_i)-\tanh(m_i(\bX))\right|\\
&\le&\sum_{i=1}^n\tanh(\mu_i)=s\tanh(B),
\end{eqnarray*}
where the second inequality follows from the subadditivity of $\tanh$. The bound \eqref{eq:bdDelta} for $\Delta_2$ then follows from the fact that $s\tanh(B)=O(\sqrt{n})$.

We now consider $\Delta_3$. Recall that $|x-\tanh(x)|\le x^2$. It suffices to show that, as $x_n\to\infty$,
\begin{equation}
\label{eq:bdDelta3}
\sup_{\bmu \in \widetilde{\Xi}(s,B)}\P_{\bQ,\bmu}\left\{\sum_{i=1}^n m_i^2(\bX)> {1\over 4}x_n\sqrt{n}\right\}\to 0,
\end{equation}
which follows from Markov inequality and the following lemma.

\begin{lemma}\label{lem:step1}
Let ${\bf X}$ be a random vector following the Ising model \eqref{eqn:general_ising}. Assume that $\mathbf{Q}_{i,j}\geq 0$ for all $(i,j)$ such that $\|\bQ\|_{\ell_\infty\to\ell_\infty}\le \rho$ for some constant $\rho<1$, and $\|\bQ\|_{\rm F}^2=O(\sqrt{n})$. Then for any fixed $C>0$,
$$\limsup_{n\rightarrow\infty}\sup_{\bmu\in [0,1]^n:\atop\sum_{i=1}^n\bmu_i\le C\sqrt{n}}\frac{1}{\sqrt{n}}\E_{\mathbf{Q},\bmu} \left(\sum_{i=1}^nm_i^2(\bX)\right)<\infty.$$
\end{lemma}
The proof of Lemma \ref{lem:step1} is deferred to the Appendix in \cite{mmy2017}.
		
\paragraph{The case of $B>1:$}
		
		In this case $s\tanh(B)\le C\sqrt{n}$ implies $s\le C'\sqrt{n}$, where $C':=C/\tanh(1)$. Also, since the statistic $\sum_{i=1}^nX_i$ is stochastically non-decreasing in $B$, without loss of generality it suffices to show that, for a fixed $S\subset[n]$ obeying $|S|=s$,
		\begin{align}\label{eqn:limit2}
		\limsup_{K\rightarrow\infty}\limsup_{n\rightarrow\infty}\limsup_{B\rightarrow\infty}\sup_{\bmu\in \tilde{\Xi}(s,B):\atop \mathrm{supp}(\bmu)=S}\P_{\mathbf{Q},\bmu}\left\{\sum_{i\in S^c}X_i>K\sqrt{n}\right\}=0.
		\end{align}
		Now, for $i\in S$ we have for $\bmu \in \tilde{\Xi}(s,B)$
		\begin{align*}\P_{\mathbf{Q},\bmu}(X_i=1|X_j=x_j,j\ne i)=\frac{e^{B+m_i(\mathbf{x})}}{e^{B+m_i(\mathbf{x})}+e^{-B-m_i(\mathbf{x})}}=&\frac{1}{1+e^{-2m_i(\mathbf{x})-2B}}
		\ge \frac{1}{1+e^{2-2B}},
		\end{align*}
		and so $\lim_{B\rightarrow\infty}\P_{\mathbf{Q},\bmu}(X_i=1,i\in S)=1$ uniformly in $\bmu \in \tilde{\Xi}(s,B)$ with  $s\leq C'\sqrt{n}$.
		Also note that for any configuration $(x_j,j\in S^c)$ we have
		\begin{align}
		\P_{\mathbf{Q},\bmu}(X_i=x_i,i\in S^c|X_i=1,i\in S)&\propto \exp\left(\frac{1}{2}\sum_{i,j\in S^c}x_ix_j\bQ_{ij}+\sum_{i\in S^c} x_i\tilde{\mu}_{S,i}\right), \label{eqn:conditional_distribution}
		\end{align}
		where $\tilde{\mu}_{S,i}:=\sum_{j\in S}\bQ_{ij}\le \|\mathbf{Q}\|_{\ell_\infty\to\ell_\infty}\le\rho$. Further we have
				\begin{align}
		\sum_{i=1}^n\tilde{\bmu}_{S,i}=\sum_{i=1}^n\sum_{j\in S}\bQ_{ij}=\sum_{j\in S}\sum_{i=1}^n\bQ_{ij}\le C'\rho \sqrt{n}. \label{eqn:mu_tilde}
		\end{align}
		We shall refer to the distribution in \eqref{eqn:conditional_distribution} as $\P_{\tilde{\bQ}_S,\tilde{\bmu}_S}$ where $\tilde{\bQ}_S$ is the $(n-s)\times (n-s)$ principle matrix of $\mathbf{Q}$ by restricting the index in $S^c$. Therefore we simply need to verify that $\tilde{\bQ}_S$ satisfy the conditions for $\bQ$ in Theorem \ref{thm:lower}. Trivially $\tilde{\bQ}_{ij}\geq 0$ for all $i,j$ and $\|\tilde{\bQ}\|_{\ell_\infty \rightarrow \ell_\infty}\leq \|{\bQ}\|_{\ell_\infty \rightarrow \ell_\infty}\leq \rho$. For verifying the third condition, i.e. 
		
$$
\left\|\tilde{\bQ}{\mathbf 1}-\frac{{\mathbf 1}^\top \tilde{\bQ}{\mathbf 1}}{n}{\bf 1}\right\|^2=O(1),
$$	

note that
\begin{align*}
O(1)=\left\|{\bQ}{\mathbf 1}-\frac{{\mathbf 1}^\top {\bQ}{\mathbf 1}}{n}{\bf 1}\right\|^2=&\frac{1}{2n}\sum_{i,j=1}^n(r_i(\mathbf{Q})-r_j(\mathbf{Q}))^2\\
\ge &\frac{1}{2n}\sum_{i,j\in S^c}(r_i(\mathbf{Q})-r_j(\mathbf{Q}))^2\\
=&\frac{n-s}{n}\times \frac{1}{2(n-s)}\sum_{i,j\in S^c}(r_i(\mathbf{Q})-r_j(\mathbf{Q}))^2\\
\ge &\frac{n-s}{n}\left\|\tilde{\bQ}{\mathbf 1}-\frac{{\mathbf 1}^\top \tilde{\bQ}{\mathbf 1}}{n}{\bf 1}\right\|^2.
\end{align*}
Therefore with $o_B(1)$ denoting a sequence of real numbers that converges to $0$ uniformly over $\bmu\in \tilde{\Xi}(s,B)$,
\begin{align*}
\ & \limsup_{B\rightarrow\infty}\sup_{\bmu\in \tilde{\Xi}(s,B):\atop \mathrm{supp}(\bmu)=S}\P_{\mathbf{Q},\bmu}\left\{\sum_{i\in S^c}X_i>K\sqrt{n}\right\}\\&\leq \limsup_{B\rightarrow\infty}\sup_{\bmu\in \tilde{\Xi}(s,B):\atop \mathrm{supp}(\bmu)=S}\left\{\P_{\mathbf{Q},\bmu}\left(\sum_{i\in S^c}X_i>K\sqrt{n}|X_j=1, j\in S\right)+o_B(1)\right\}\\
&=\limsup_{B\rightarrow\infty}\sup_{\bmu\in \tilde{\Xi}(s,B):\atop \mathrm{supp}(\bmu)=S}\P_{\tilde{\bQ}_S,\tilde{\bmu}_S}\left(\sum_{i\in S^c}X_i>K\sqrt{n}\right)\\
& \leq \sup_{S\subset [n]}\sup_{\tilde{\bmu}_S:\atop \sum\limits_{i \in S^c}\tilde{\mu}_{S,i}\leq C'\rho\sqrt{n}}\P_{\tilde{\bQ}_S,\tilde{\bmu}_S}\left(\sum_{i\in S^c}X_i>K\sqrt{n}\right),
\end{align*}
	where the last line follows from \eqref{eqn:mu_tilde}. The proof of the claim \eqref{eqn:limit2} thereafter follows using the same argument as that for the case when $B<1$ since $\tilde{\mu}_{S,i}\leq \rho<1$ for each $i\in S^c$.
		
\subsubsection{Proof of \eqref{eq:claim_lower}:}

It is clear that, by symmetry,
\begin{equation}
\label{eq:eq00}
\P_{\mathbf{Q},\mathbf{0}}\Big(\Big|\sum_{i=1}^n X_i|>K\sqrt{n}\Big)=2\P_{\mathbf{Q},\mathbf{0}}\Big(\sum_{i=1}^n X_i> K\sqrt{n}\Big).
\end{equation}
In establishing \eqref{eq:claim_upper}, we essentially proved that
\begin{equation}
\label{eq:eq02}
\limsup_{K\rightarrow\infty}\limsup_{n\rightarrow\infty}\sup_{\bmu\in \tilde{\Xi}(s,B)}\P_{\mathbf{Q},\bmu}\left(\sum_{i=1}^nX_i>K\sqrt{n}\right)=0.
\end{equation}
By choosing $K$ large enough, we can make the right hand side of \eqref{eq:eq00} less than $1/2$. This gives 
\begin{align}\label{eq:re-write}
\sum_{{\bf x}\in \{-1,1\}^n}e^{{\bf x}^\top\mathbf{Q}{\bf x}/2}\le 2\sum_{{\bf x}\in D_{n,K}}e^{{\bf x}^\top\mathbf{Q}{\bf x}/2},
\end{align}
where  $D_{n,K}:=\Big\{|\sum_{i=1}^n X_i|\le K\sqrt{n}\Big\}$. Then, setting $C_n:=\{\sum_{i=1}^nX_i>\lambda \sqrt{n}\}$, for any $K>\lambda$ we have 
\begin{align*}
\P_{\mathbf{Q},\mathbf{0}}(C_n)
\ge \P_{\mathbf{Q},\mathbf{0}}(C_n\cap D_{n,K})
=&\frac{\sum_{{\bf x}\in C_n\cap D_{n,K}} e^{{\bf x}'\mathbf{Q}{\bf x}/2}}{\sum_{{\bf x}\in \{-1,1\}^n} e^{{\bf x}'\mathbf{Q}{\bf x}/2}}\\
\ge &\frac{1}{2}\frac{\sum_{{\bf x}\in C_n\cap D_{n,K}} e^{{\bf x}'\mathbf{Q}{\bf x}/2}}{\sum_{{\bf x}\in D_{n,K}} e^{{\bf x}'\mathbf{Q}{\bf x}/2}}\\
\ge &\frac{e^{-2Kt}}{2}\frac{\sum_{{\bf x}\in C_n\cap D_{n,K}} e^{{\bf x}'\mathbf{Q}{\bf x}/2+\frac{t}{\sqrt{n}}\sum_{i=1}^nx_i}}{\sum_{{\bf x}\in D_{n,K}} e^{{\bf x}'\mathbf{Q}{\bf x}/2}}\\
=&\frac{e^{-2Kt}}{2}\frac{\P_{\mathbf{Q},\bmu(t)}(C_n\cap D_{n,K})}{\P_{\mathbf{Q},\mathbf{0}}( D_{n,K})}\frac{Z(\mathbf{Q},\bmu(t))}{Z(\mathbf{Q},{\bf 0})}\\
\ge &\frac{e^{-2Kt}}{2}\P_{\mathbf{Q},\bmu(t)}(C_n\cap D_{n,K}),
\end{align*}
where $\bmu(t)=tn^{-1/2}{\bf 1}$. In the last inequality we use the fact that the function $t\mapsto Z(\mathbf{Q},\bmu(t))$ is non-increasing in $t$ on $[0,\infty)$, as
$$\frac{\partial }{\partial t}Z(\mathbf{Q},\bmu(t))=\frac{1}{\sqrt{n}}\E_{\mathbf{Q},\bmu(t)}\sum_{i=1}^nX_i\ge \frac{1}{\sqrt{n}}\E_{\mathbf{Q},{\bf 0}}\sum_{i=1}^nX_i=0. $$
 To show \eqref{eq:claim_lower}, it thus suffices to show that there exists $K$ large enough and $t>0$ such that
$$
\liminf_{n\rightarrow\infty}\P_{\mathbf{Q},\bmu(t)}(C_n\cap D_{n,K})>0.$$
To this end, it suffices to show that for any $\lambda>0$ there exists $t$ such that
\begin{equation}
\label{eq:eq01}
\liminf_{n\rightarrow\infty}\P_{{\mathbf{Q}},\bmu(t)}(\sum_{i=1}^nX_i>\lambda \sqrt{n})>0.
\end{equation}
If \eqref{eq:eq01} holds, then there exists $t>0$ such that 
$$\liminf_{n\rightarrow\infty}\P_{\mathbf{Q},\bmu(t)}(C_n)>0.$$
It now suffices to show that for any $t$ fixed one has
\begin{align*}
\limsup_{K\rightarrow\infty}\limsup_{n\rightarrow\infty}\P_{\mathbf{Q},\bmu(t)}(D_{n,K}^c)=0,
\end{align*}
which follows from \eqref{eq:eq02}.

It now remains to show \eqref{eq:eq01}. To begin, note that for $h>0$,
\begin{eqnarray*}
\E_{{\mathbf{Q}},\bmu(h)} X_i&=&\E_{{\mathbf{Q}},\bmu(h)} \tanh\left(m_i(\bX)+\frac{h}{\sqrt{n}}\right)\\
&=&\E_{{\mathbf{Q}},\bmu(h)}\frac{\tanh(m_i(\bX))+\tanh\left(\frac{h}{\sqrt{n}}\right)}{1+\tanh(m_i(\bX))\tanh\left(\frac{h}{\sqrt{n}}\right)}\\
&\ge& \frac{1}{2}\left[\E_{{\mathbf{Q}},\bmu(h)} \tanh(m_i(\bX))+\tanh\left(\frac{h}{\sqrt{n}}\right)\right]\\
&\ge& \frac{1}{2}\tanh\left(\frac{h}{\sqrt{n}}\right).
\end{eqnarray*}
In the last inequality we use Holley inequality  \citep[e.g., Theorem 2.1 of][]{grimmett2006random} for the two probability measures $\P_{\mathbf{Q},\mathbf{0}}$ and $\P_{\mathbf{Q},\bmu(h)}$ to conclude $$\E_{\mathbf{Q},\bmu(h)}\tanh(m_i(\bX)\ge \E_{\mathbf{Q},0}\tanh(m_i(\bX))=0,$$
in the light of (2.7) of \cite{grimmett2006random}. Adding over $1\le i\le n$ gives
\begin{align}\label{eq:mean_estimate}
F_n'(h)=\frac{1}{\sqrt{n}}\E_{{\mathbf{Q},\bmu(h)} }\sum_{i=1}^n X_i\ge \frac{\sqrt{n}}{2}\tanh\left(\frac{h}{\sqrt{n}}\right),
\end{align}
where $F_n(h)$ is the log normalizing constant for the model $\P_{\mathbf{Q},\bmu(h)}$. 
Thus, using  Markov's inequality one gets
\begin{align*}
\P_{\mathbf{Q},\bmu(t)}\left(\sum_{i=1}^nX_i\le \lambda \sqrt{n}\right)
=&\P_{\mathbf{Q},\bmu(t)}\left(e^{- \frac{1}{\sqrt{n}}\sum_{i=1}^nX_i}\ge e^{- \lambda}\right)
\le \exp\left\{ \lambda+F_n(t- 1)-F_n(t)\right\},
\end{align*}
Using \eqref{eq:mean_estimate}, the exponent in the rightmost hand side can be estimated as
\begin{align*}
 \lambda +F_n(t- 1)-F_n(t)
= \lambda -\int_{t-1}^{t} F_n'(h)dh\le  \lambda-\frac{\sqrt{n}}{2}\tanh\left(\frac{t-1}{\sqrt{n}}\right),
\end{align*}
which is negative and uniformly bounded away from $0$ for all $n$ large for $t=4\lambda+1$, from which \eqref{eq:eq01} follows.

\subsection{Proof of Theorem \ref{thm:z1_unstructured}}
We set $m_i(\bX)=\sum_{j=1}^n \bQ_{ij}X_j$ and assume $\bmu\in \tilde{\Xi}(s,B)$ with $s\tanh(B)\le C\sqrt{n}$.  By the same argument as that of Section \ref{likelihood_ratio}, it suffices to show that there does not exist a sequence of positive reals $\{L_n\}_{n\ge 1}$ such that
$$
\P_{\mathbf{Q},\mathbf{0}}\left(\sum_{i=1}^nX_i>L_n\right)+\P_{\mathbf{Q},\bmu}\left(\sum_{i=1}^nX_i<L_n\right)\to 0.
$$			
Suppose, to the contrary, that there exists such a sequence. For any $t\in \R$ we have
	\begin{align*}
	\E_{\mathbf{Q},\mathbf{0}}\exp\left\{\frac{t}{\sqrt{n}}\sum_{i=1}^nX_i\right\}=&\frac{Z\left(\mathbf{Q},\frac{t}{\sqrt{n}}{\bf 1}\right)}{Z\left(\mathbf{Q},\mathbf{0}\right)}
	=\lambda_1\left(\frac{t}{\sqrt{n}}\right)^n+\lambda_2\left(\frac{t}{\sqrt{n}}\right)^n,
	\end{align*}
	where $$\lambda_i(t):=\frac{e^\theta \cosh( t)+(-1)^{i+1}\sqrt{e^{2\theta}\sinh(t)^2+e^{-2\theta}}}{e^\theta+e^{-\theta}}.$$
	This computation for the normalizing constants for the Ising model on the cycle graph of length $n$ is standard \citep{ising1925beitrag}.
	By a direct calculation we have $$\lambda_1(0)=1> \lambda_2(0)=\tanh(\theta),\quad \lambda_1'(0)=\lambda_2'(0)=0,\quad c(\theta):=\lambda_1''(0)>0,$$
	and so 
	\begin{align*}
	\E_{\mathbf{Q},\mathbf{0}}e^{\frac{t}{\sqrt{n}}\sum_{i=1}^nX_i}&
	=\lambda_1\Big(\frac{t}{\sqrt{n}}\Big)^n+\lambda_2\Big(\frac{t}{\sqrt{n}}\Big)^n\stackrel{n\rightarrow\infty}{\rightarrow} e^{\frac{c(\theta)t^2}{2)}}.
	\end{align*}
	This implies that under $H_0$ $$\frac{1}{\sqrt{n}}\sum_{i=1}^nX_i\stackrel{d}{\rightarrow}N(0,{c(\theta)}),$$
	which for any $\lambda>0$ gives
	\begin{align*}
	\liminf_{n\rightarrow\infty}\P_{\mathbf{Q},\mathbf{0}}\left(\sum_{i=1}^nX_i>\lambda \sqrt{n}\right)>0.
	\end{align*}
	Therefore, $L_n\gg \sqrt{n}$. Now invoking Lemma \ref{lemma_chatterjee}, for any $K>0$ we have
	$$\P_{\mathbf{Q},\bmu}\left\{\left|\sum_{i=1}^n(X_i-\tanh(m_i(\bX)+\mu_i)\right|>K\sqrt{n}\right\}\le 2e^{-K^2/4(1+\theta)^2}.$$
	On this set we have for a universal constant $C<\infty$
	\begin{align*}
	\left|\sum_{i=1}^n(X_i-\tanh(m_i(\bX))\right|\le &\left|\sum_{i=1}^n(X_i-\tanh(m_i(\bX)+\mu_i))\right|\\
	&\hskip 50pt +\left|\sum_{i=1}^n(\tanh(m_i(\bX)+\mu_i)-\tanh(m_i(\bX)))\right|\\
	\le &K\sqrt{n}+C\sum_{i=1}^n\tanh(\mu_i)\\
	\le &K\sqrt{n}+Cs\tanh(B),
	\end{align*}
	and so
	\begin{align}\label{eq:raj4}
	\P_{\mathbf{Q},\bmu}\left\{\left|\sum_{i=1}^n(X_i-\tanh(m_i(\bX)))\right|>K\sqrt{n}+Cs\tanh(B)\right\}\le 2e^{-K^2/4(1+\theta)^2}.
	\end{align}
	Also, setting $g(t):=t/\theta-\tanh(t)$, we get
	\begin{align*}
	\sum_{i=1}^n(X_i-\tanh(m_i(\bX))=&\sum_{i=1}^ng(m_i(\bX))=\{Q_n(\bX)-R_n(\bX)\} g(\theta),
	\end{align*}
	where $$Q_n(\bX):=\left|\{1\le i\le n:m_i(\bX)=\theta\}\right|,\quad R_n(\bX):=\left|\{1\le i\le n:m_i(\bX)=-\theta\}\right|.$$
	Indeed, this holds, as in this case $m_i(\bX)$ can take only three values $\{-\theta,0,\theta\}$, and $g(.)$ is an odd function. Thus using \eqref{eq:raj4} gives
	\begin{align*}
	\P_{\mathbf{Q},\bmu_n}\left\{\left|Q_n(\bX)-R_n(\bX)\right|>\frac{K\sqrt{n}+Cs\tanh(B)}{g(\theta)}\right\}\le 2e^{-K^2/4(1+\theta)^2}.
	\end{align*}
	But then we have
	\begin{align*}
	\P_{\mathbf{Q},\bmu_n}\left\{\sum_{i=1}^nX_i>L_n\right\}=&\P_{\mathbf{Q},\bmu}\left\{\sum_{i=1}^nm_i(\bX)>\theta L_n\right\}\\
	=&\P_{\mathbf{Q},\bmu}\left\{Q_n(\bX)-R_n(\bX)>L_n\right\}\le 2e^{-K^2/4(1+\theta)^2},
	\end{align*}
	as $$L_n\gg \frac{K\sqrt{n}+Cs\tanh(B)}{g(\theta)}.$$
	This immediately yields the desired result.

\subsection{Proof of Theorem \ref{thm:curie}}
By Theorem \ref{thm:lower}, there is no asymptotically powerful test if $s\tanh(B)=O(n^{1/2})$. It now suffices to show that the na\"ive test is indeed asymptotically powerful. To this end, we first consider the Type I error. By Theorem 2 of \cite{Ellis_Newman},
$$
\sqrt{n}\bar{X}\to_dN\left(0,\frac{1}{1-\theta}\right),
$$
which immediately implies that Type I error
$$
\P_{\theta, \mathbf{0}}\left(\sqrt{n}\bar{X}\ge L_n\right)\to 0.
$$

Now consider Type II error. Observe that
\begin{eqnarray*}
\bar{X}-f_{\bQ,\bmu}(\bX)&=&{1\over n}\sum_{i=1}^n \tanh\left(\sum_{j\neq i}\bQ_{ij}X_j+\mu_i\right)\\
&=&{1\over n}\sum_{i=1}^n \tanh\left(\theta\bar{X}+\mu_i-\theta X_i/n\right)\\
&=&{1\over n}\sum_{i=1}^n \tanh\left(\theta\bar{X}+\mu_i\right)+O(n^{-1}),
\end{eqnarray*}
where the last equality follows from the fact that $\tanh$ is Lipschitz. In addition,
\begin{eqnarray*}
{1\over n}\sum_{i=1}^n \tanh\left(\theta\bar{X}+\mu_i\right)&=&\tanh\left(\theta\bar{X}\right)+{1\over n}\sum_{i\in {\rm supp}(\bmu)}\left[\tanh\left(\theta\bar{X}+\mu_i\right)-\tanh\left(\theta\bar{X}\right)\right]\\
&\ge&\tanh\left(\theta\bar{X}\right)+{1\over n}\sum_{i\in {\rm supp}(\bmu)}\left[\tanh\left(\theta\bar{X}+B\right)-\tanh\left(\theta\bar{X}\right)\right]\\
&\ge&\tanh\left(\theta\bar{X}\right)+{s\tanh(B)\over n}\left[1-\tanh\left(\theta\bar{X}\right)\right],\\
&\ge&\tanh\left(\theta\bar{X}\right)+{s\tanh(B)\over n}\left[1-\tanh\left(\theta\right)\right],
\end{eqnarray*}
where the second to last inequality follows from (\ref{eq:lowertanh}). In other words,
$$
\sqrt{n}(\bar{X}-\tanh(\theta\bar{X}))-\sqrt{n}f_{\bQ,\bmu}(\bX)\ge {s\tanh(B)\over \sqrt{n}}\left[1-\tanh\left(\theta\right)\right].
$$
Since $\sup_{x\in\R}\frac{x-\tanh(\theta x)}{x}<\infty,$  an application of Lemma \ref{lemma_chatterjee}, together with the fact that $L_n=o(n^{-1/2}s\tanh(B))$ yields
$$
\P_{\theta,\bmu}\left(\sqrt{n}\bar{X}\ge L_n\right)\to 1.
$$

\subsection{Proof of Theorem \ref{thm:curie1}}
The proof of attainability follows immediately from Theorem \ref{thm:upper}. Therefore here we focus on the proof of the lower bound. As before, by the same argument as those following Section \ref{likelihood_ratio}, it suffices to show that there does not exist a sequence of positive reals $\{L_n\}_{n\ge 1}$ such that
$$
\P_{\mathbf{Q},\mathbf{0}}\left(\sum_{i=1}^nX_i>L_n\right)+\P_{\mathbf{Q},\bmu}\left(\sum_{i=1}^nX_i<L_n\right)\to 0.
$$			

From the proof of Lemma \ref{lemma_chatterjee} and the inequality $|\tanh(x)-\tanh(y)|\le |x-y|$, for any fixed $t<\infty$ and $\bmu\in \widetilde{\Xi}(s,B)$ we have
\begin{align*}
\Ptheta\Big(\bar{X}>\frac{s}{n}\tanh(\theta \bar{X}+B)+\frac{n-s}{n}\tanh(\theta \bar{X})+\frac{\theta}{n}+\frac{t}{\sqrt{n}}\Big)\le  2 e^{-\frac{t^2}{2na_n}},
\end{align*}
where
$$a_n:=\frac{2}{n}+\frac{2\theta}{n}+\frac{2\theta}{n^2}.$$
Also note that
\begin{align*}
&\frac{s}{n}\tanh(\theta \bar{X}+B)+\frac{n-s}{n}\tanh(\theta \bar{X})\le \tanh(\theta \bar{X})+C\frac{s}{n}\tanh(B),
\end{align*}
for some constant $C<\infty$. Therefore
\begin{align*}
\Ptheta\left\{\bar{X}-\tanh(\theta \bar{X})> C\frac{s}{n}\tanh(B)+\frac{\theta}{n}+\frac{t}{\sqrt{n}}\right\}\le 2\exp\left(-t^2/2na_n\right).
\end{align*}

Since $s\tanh(B)=O(n^{1/2})$, we have
\begin{align}\label{eq:foralltheta}\sup_{\bmu\in \widetilde{\Xi}(s,B)}\Ptheta\left\{\bar{X}-\tanh(\theta \bar{X})> \frac{C(t)}{\sqrt{n}}\right\}\le 2 \exp\left(-t^2/2na_n\right)
\end{align}
for some finite positive constant $C(t)$. Now, invoking Theorem 1 of \cite{Ellis_Newman}, under $H_0:\bmu=\mathbf{0}$ we have
$$\sqrt{n}(\bar{X}-m)|\bar{X}>0\stackrel{d}{\rightarrow}N\left(0,\frac{1-m^2}{1-\theta(1-m^2)}\right),$$
where $m$ is the unique positive root of $m=\tanh(\theta m)$. The same argument as that from Section \ref{likelihood_ratio} along with the requirement to control the Type I error then imply that without loss of generality one can assume the test $\phi_n$ rejects if $\bar{X}> m+L_n$, where $L_n\gg n^{-1/2}$. 

Now, note that $g(x)=x-\tanh(\theta x)$ implies that  $g'(x)$ is positive and increasing on the set $[m,\infty)$, and therefore 
$$g(x)\ge g(m)+(x-m)g'(m).$$
This gives
\begin{align*}&\P_{\theta,\bmu}\left(\bar{X}> m+L_n, \bar{X}-\tanh(\theta \bar{X})\le \frac{C(t)}{\sqrt{n}}\right)\\
\le &\P_{\theta,\bmu}\left(\bar{X}> m+L_n, \bar{X}-m\le \frac{C(t)}{g'(m)\sqrt{n}}\right),
\end{align*}
which is $0$ for all large $n$, as $L_n\gg n^{-1/2}$. This, along with \eqref{eq:foralltheta} gives $$\liminf_{n\rightarrow\infty}\inf_{\bmu\in \widetilde{\Xi}(s,B)}\E_{\theta,\bmu}(1-\phi_n)\ge1,$$  thus concluding the proof.

\subsection{Proof of Theorem \ref{thm:corr}}

The proof of Theorem \ref{thm:corr} is based on an auxiliary variable approach known as Kac's Gaussian transform  \citep{kac1959partition}, which basically says that the moment generating function of $N(0,1)$ is $e^{t^2/2}$. This trick has already been used in computing asymptotics of log partition functions \citep{comets1991asymptotics,park2004solution,mukherjee2013consistent}.

In particular, the proof relies on the following two technical lemmas. The proof to both lemmas is relegated to the Appendix in \cite{mmy2017} for brevity.

\begin{lem}\label{lem:bayesian}
Let ${\bX}$ follow a Curie-Weiss model of \eqref{eq:Curie} with $\theta>0$. Given $\bX=\bx$ let $Z_n$ be a normal random variable with mean $\bar{x}$ and variance $1/(n\theta)$. Then
\begin{enumerate}
\item[(a)]
Given $Z_n=z$ the random variables $(X_1,\cdots,X_n)$ are mutually independent,  with
\begin{align*}
\Ptheta(X_i=x_i)=&\frac{e^{(\mu_i+z\theta)x_i}}{e^{\mu_i+z\theta}+e^{-\mu_i-z\theta}},
\end{align*}
where $x_i\in \{-1,1\}$.
\item[(b)]
The marginal density of $Z_n$ is proportional to $e^{-f_{n,\bmu}(z)}$, where
\begin{align}\label{eqn:z_density}
f_{n,\bmu}(z):=\frac{n\theta z^2}{2}-\sum_{i=1}^n\log\cosh(\theta z+\mu_i).
\end{align}
\item[(c)] 
$$\sup_{\bmu\in [0,\infty)^n}\Etheta\Big(\sum_{i=1}^n (X_i-\tanh(\mu_i+\theta Z_n))\Big)^2\le n.$$
\end{enumerate}
\end{lem}

While the previous lemma applies to all $\theta>0$, the next one specializes to the case $\theta=1$ and gives crucial estimates which will be used in proving Theorem \ref{thm:corr}.

For any $\bmu\in (\R^+)^n$ set
$$A(\bmu):=\frac{1}{n}\sum_{i=1}^n\tanh(\bmu_i).$$
This can be thought of as the total amount of signal present in the parameter $\bmu$. 
In particular, note that for $\bmu\in \Xi(s,B)$ we have $$A(\bmu)\ge \frac{s\tanh(B)}{n},$$
and for $\bmu\in \tilde{\Xi}(s,B)$ we have $$A(\bmu)= \frac{s\tanh(B)}{n}.$$
In the following we abbreviate $s\tanh(B)/n:=A_n$.

\begin{lem}\label{lem:theta1}
\begin{enumerate}
\item[(a)]
If $\theta=1$, for any $\bmu\in \Xi(s,B) $ the function $f_{n,\bmu}(\cdot)$ defined by \eqref{eqn:z_density} is strictly convex, and has a unique global minimum $m_n\in (0,1]$, such that
\begin{align}\label{eq:est1}
m_n^3=\Theta(A(\bmu)).
\end{align}
\item[(b)]
 $$\limsup_{K\rightarrow\infty}\limsup_{n\rightarrow\infty}\Ptheta(Z_n-m_n>K n^{-1/4})=0.$$
\item[(c)]
If $A_n\gg n^{-3/4}$ then there exists $\delta>0$ such that
$$\limsup_{n\rightarrow\infty}\sup_{\bmu:A(\bmu)\ge A_n}\P_{\theta,\bmu}\Big(Z_n\le \delta m_n\Big)=0.$$
\end{enumerate}
\end{lem}

The proof of Lemma \ref{lem:theta1} can be found in the Appendix in \cite{mmy2017}. We now come back to the proof of Theorem \ref{thm:corr}. To establish the upper bound, define a test function $\phi_n$ by $\phi_n({\bf X})=1$ if $\bar{X}>2\delta A_n^{1/3}$, and $0$ otherwise, where $\delta$ is as in part (c) of Lemma \ref{lem:theta1}.
By Theorem 1 of \cite{Ellis_Newman}, under $H_0:\bmu=0$ we have
\begin{align}\label{eq:null}
n^{1/4}\bar{X}\stackrel{d}{\rightarrow}Y,
\end{align}
where $Y$ is a random variable on $\R$ with density proportional to $e^{-y^4/12}$. Since $A_n\gg n^{-3/4}$ we have
$$\Pzero(\bar{X}>2\delta A_n^{1/3})=o(1),$$
and so it suffices to show that
\begin{align}
\sup_{\bmu:A(\bmu)\ge A_n}\Ptheta(\bar{X}\le 2\delta A_n^{1/3 })=o(1). \label{eqn:theta1_up_to_show}
\end{align}
To this effect, note that
\begin{align*}
\sumn X_i=&\sumn (X_i-\tanh(\mu_i+Z_n))+\sumn \tanh(\mu_i+Z_n)\\
\ge &\sumn(X_i-\tanh(\mu_i+Z_n))+n \tanh(Z_n)
\end{align*}
Now by Part (c) of Lemma \ref{lem:bayesian} and Markov inequality,
$$|\sumn (X_i-\tanh(\mu_i+Z_n))|\leq \delta nA_n^{1/3}$$
with probability converging to $1$ uniformly over $\bmu\in [0,\infty)^n$. Thus it suffices to show that
$$\sup_{\bmu:A(\bmu)\ge A_n}\Ptheta(nZ_n\le 3\delta n A_n^{1/3})=o(1).$$
But this follows on invoking Parts (a) and (c) of  Lemma \ref{lem:theta1}, and so the proof of the upper bound is complete.

To establish the lower bound, by the same argument as that from Section \ref{likelihood_ratio}, it suffices to show that there does not exist a sequence of positive reals $\{L_n\}_{n\ge 1}$ such that
$$
\P_{\mathbf{Q},\mathbf{0}}\left(\sum_{i=1}^nX_i>L_n\right)+\P_{\mathbf{Q},\bmu}\left(\sum_{i=1}^nX_i<L_n\right)\to 0.
$$			
If $\lim_{n\rightarrow \infty}n^{-3/4}L_n<\infty$, then \eqref{eq:null} implies
$$\liminf_{n\rightarrow\infty}\Ezero\phi_n>0,$$
and so we are done. Thus assume without loss of generality that $n^{-3/4}L_n\rightarrow\infty$. In this case we have
\begin{align*}
\sum_{i=1}^nX_i=&\sum_{i=1}^n(X_i-\tanh(\mu_i+Z_n))+\sum_{i=1}^n\tanh(\mu_i+Z_n)\\
\le &\sum_{i=1}^n(X_i-\tanh(\mu_i+Z_n))+\sum_{i=1}^n\tanh(\mu_i)+n|Z_n|, 
\end{align*}
and so
\begin{align*}
&\Ptheta\left(\sum_{i=1}^nX_i>L_n\right)\\
\le& \Ptheta\left\{|\sum_{i=1}^nX_i-\tanh(\mu_i+Z_n)|>L_n/3\right\}+\Ptheta\left\{nZ_n>L_n/3\right\}+\Ptheta\left\{nZ_n<-L_n/3\right\}
\end{align*}
where we use the fact that
$$\sum_{i=1}^n\tanh(\mu_i)=O(n^{1/4})\ll L_n.$$
Now by Part (c) of Lemma \ref{lem:bayesian} and Markov inequality, the first term above converges to $0$ uniformly over all $\bmu$.  Also by Parts (a) and (b) of Lemma \ref{lem:theta1}, $\Ptheta\left\{nZ_n>L_n/3\right\}$ converges to $0$ uniformly over all $\bmu$ such that $A(\bmu)=O(n^{-3/4})$. Finally note that the distribution of $Z_n$ is stochastically increasing in $\bmu$, and so
$$\Ptheta\left\{nZ_n<-L_n/3\right\}\le \P_{\theta,{\bf 0}}\left\{nZ_n<-L_n/3\right\},$$
which converges to $0$ by \eqref{eq:null}. This completes the proof of the lower bound.

\section*{Acknowledgments} The authors thank the Associate Editor and two anonymous referees for numerous helpful comments which substantially improved the content and presentation of the paper. The authors also thank James Johndrow for helping with package \verb+rstan+.

\begin{supplement} 
	\stitle{Supplement to "Global Testing Against Sparse Alternatives under Ising Models"}
	\slink[doi]{COMPLETED BY THE TYPESETTER}
	\sdatatype{.pdf}
	\sdescription{The supplementary material contain the proofs of additional technical results.}
\end{supplement}

\bibliographystyle{imsart-nameyear}
\bibliography{biblio_ising_minimax_new}
\appendix
\section*{Appendix -- Proof of Auxiliary Results}\label{sec:technical_lemmas}

\begin{proof}[Proof of Lemma \ref{lemma_chatterjee}]
This is a standard application of Stein's Method for concentration inequalities \citep{chatterjee2005concentration}. \tcred{The details are included here for completeness.}
One begins by noting that 
\begin{align*}
\EJ\left(X_i|X_j,j \neq i\right)=\tanh\left( m_i\left(\bX\right)+\mu_i\right), \quad m_i\left(\bX\right):=\sum\limits_{j =1}^n\bQ_{ij}X_j.
\end{align*}
Now let $\bX$ be drawn from \eqref{eqn:general_ising} and let $\bX^{'}$ is drawn by moving one step in the Glauber dynamics, i.e. let $I$ be a random variable which is discrete uniform on $\{1,2,\cdots,n\}$,  and replace the $I^{th}$ coordinate
of $\bX$ by an element drawn from the conditional distribution of the $I^{th}$ coordinate given the rest. It is not difficult to see that $(\bX,\bX^{'})$ is an exchangeable pair of random vectors. Further define an anti-symmetric function $F: \mathbb{R}^n \times \mathbb{R}^n \rightarrow \mathbb{R}$ as $F(\bx,\by)=\sum_{i=1}^n \left(x_{i}-y_{i}\right)$, which ensures that 
\begin{align*}\EJ\left(F(\bX,\bX^{'})|\bX\right)
=&\frac{1}{n}\sum\limits_{j=1}^nX_j-\tanh\left( m_j\left(\bX\right)+\mu_j\right)=f_{\bmu}(\bX).
\end{align*}
 Denoting $\bX^i$ to be $\bX$ with $X_i$ replaced by $-X_i$, by Taylor's series we have
\begin{align*}
&\tanh(m_j(\bX^i)+\mu_j)-\tanh(m_j(\bX)+\mu_j)\\
=&(m_j(\bX^i)-m_j(\bX))g'(m_j(\bX))+\frac{1}{2}(m_j(\bX^i)-m_j(\bX))^2 g''(\xi_{ij})\\
=&-2\bQ_{ij}X_ig'(m_j(\bX))+2\bQ_{ij}^2 g''(\xi_{ij})
\end{align*}
 for some $\{\xi_{ij}\}_{1\le i,j\le n}$, where $g(t)=\tanh(t)$. Thus $f_{\bmu}(\bX)-f_{\bmu}(\bX^i)$ can be written as   
 \begin{align*}
 f_{\bmu}(\bX)-f_{\bmu}(\bX^i)
  = &\frac{2X_i}{n}+\frac{1}{n}\sum\limits_{j=1}^n\Big\{\tanh\left( m_j\left(\bX^i\right)+\mu_j\right)-\tanh\left( m_j\left(\bX\right)+\mu_j\right)\Big\}\\
 = & \frac{2X_i}{n}-\frac{2X_i}{n}\sum_{j=1}^n  \bQ_{ij} g'(m_j(\bX))+\frac{2}{n}\sum_{j=1}^n  \bQ_{ij}^2 g''(\xi_{ij})
 \end{align*}
 \tcred{Now setting $p_i(\bX):=\PJ(X_i'=-X_i|X_k,k\ne i)$ we have
 \begin{align*}
 v(\bX):=&\frac{1}{2}\EJ \Big(|f_{\bmu}(\bX)-f_{\bmu}(\bX')\|(X_I-X_I')|\Big|{\bf X}\Big)\\
 =&\frac{1}{n}\sum_{i=1}^n |f_{\bmu}(\bX)-f_{\bmu}(\bX^i)|p_i(\bX)\\
&\leq\frac{2}{n^2}\sum_{i=1}^n  p_i(\bX)-\frac{2}{n^2} \sum_{i,j=1}^n|\bQ_{ij}p_i(\bX)g'(m_j(\bX))|\\&+\frac{2}{n^2}\sum_{i,j=1}^n\bQ_{ij}^2g''(\xi_{ij})^2X_ip_i(\bX)\\
 \le &\frac{2}{n}+\frac{2}{n^2}\sup_{{\bf u},{\bf v}\in [0,1]^n}|{\bf u}'\mathbf{Q}{\bf v}|+\frac{2}{n^2}\sum_{i,j=1}^n\bQ_{ij}^2,
 \end{align*}
where in the last line we use the fact that $\max(|g'(t)|,|g''(t)|)\le 1$. 
The proof of the Lemma is then completed by an application of Theorem 3.3 of \cite{chatterjee2007stein}.}
\end{proof}

\begin{proof}[Proof of Lemma \ref{lem:step1}]

Let $\mathbf{Y}:=(Y_1,\cdots,Y_n)$ be i.i.d. random variables on $\{-1,1\}$ with $\P(Y_i=\pm1)=\frac{1}{2}$, and let  $\mathbf{W}:=(W_1,\cdots,W_n)\stackrel{i.i.d.}{\sim}N(0,1)$. Also, for any $t>0$ let $Z(t{\bf Q},\mu)$ denote the normalizing constant of the p.m.f.
$$\frac{1}{Z(t\mathbf{Q}, \bmu)}\exp\left(\frac{1}{2}\bx^\top t\mathbf{Q} \bx+\bmu^\top\bx\right)$$
Thus we have
\begin{align*}
2^{-n}Z(t\mathbf{Q},\bmu)=\E \text{exp}\Big(\frac{t}{2}\mathbf{Y}^\top\mathbf{Q}\mathbf{Y}+\sum_{i=1}^n\mu_iY_i\Big)\le \E  \text{exp}\Big(\frac{t}{2}\mathbf{W}^\top\mathbf{Q}\mathbf{W}+\sum_{i=1}^n\mu_iW_i\Big),
\end{align*}
where we use the fact that $\E Y_i^k\le \E W_i^k$ for all positive integers $k$. Using spectral decomposition write $\mathbf{Q}=\mathbf{P}^\top\mathbf{\Lambda}\mathbf{ P}$ and set $\mathbf{\nu}:=\mathbf{P}\bmu, \widetilde{\mathbf{W}}=\mathbf{P}\mathbf{W}$ to note that
\begin{align*}
\E \text{exp}\Big(\frac{t}{2}\mathbf{W}^\top\mathbf{Q}\mathbf{W}+\sum_{i=1}^n\mu_iW_i\Big)=\E \text{exp}\Big(\frac{t}{2}\sum_{i=1}^n\lambda_i\widetilde{W}_i^2+\sum_{i=1}^n\nu_i\widetilde{W}_i\Big)
=\prod_{i=1}^n\frac{e^{\frac{\nu_i^2}{2(1-t\lambda_i)}}}{\sqrt{1-t\lambda_i}} .
\end{align*}
Combining for any $t>1$ we have the bounds 
\begin{align}\label{eqn:bound}
2^n \prod_{i=1}^n\cosh(\mu_i)=Z(\mathbf{0},\bmu)\le Z(\mathbf{Q},\bmu)\le Z(t\mathbf{Q},\bmu)\le 2^n 
\frac{ e^{\sum_{i=1}^n\frac{\nu_i^2}{2(1-t\lambda_i)}}}{\prod_{i=1}^n\sqrt{1-t\lambda_i}},
\end{align}
where the lower bound follows from on noting that $\log Z(t\mathbf{Q},\bmu)$ is monotone non-decreasing in $t$, using results about exponential families. Thus invoking convexity of the function $t\mapsto \log Z(t\mathbf{Q},\bmu)$ we have
\begin{align*}
\ & \E_{\mathbf{Q},\bmu} \frac{1}{2} {\bX}^\top\mathbf{Q}\bX=\frac{\partial \log Z(t\mathbf{Q},\bmu)}{\partial t}\Big|_{t=1}
\\&\le\frac{\log Z(t\mathbf{Q},\bmu)-\log Z(\mathbf{Q},\bmu)}{t-1}\\
&\le \sum_{i=1}^n \Big\{\frac{\nu_i^2}{2(1-t\lambda_i)}-\log\cosh(\mu_i)\Big\}-\sum_{i=1}^n\frac{1}{2}\log(1-t\lambda_i),
\end{align*}
where we use the bounds obtained in \eqref{eqn:bound}.
Proceeding to bound the rightmost hand side above, set $t=\frac{1+\rho}{2\rho}>1$ and note that $$|t\lambda_i|\le \frac{1+\rho}{2}<1.$$
For $x\in \frac{1}{2}[-(1+\rho),(1+\rho)]\subset (-1,1)$ there exists a constant $\gamma_\rho<\infty$ such that
$$\frac{1}{1-x}\le 1+x+2\gamma_\rho x^2,\quad -\log(1-x)\le x+2\gamma_\rho x^2.$$ Also a Taylor's expansion gives
$$-\log \cosh(x)\le -\frac{x^2}{2}+x^4,$$ where we have used the fact that $\|(\log\cosh(x))^{(4)}\|_\infty\le 1.$ These, along with the observations that $$\sum_{i=1}^n\lambda_i=tr(\mathbf{Q})=0, \quad \sum_{i=1}^n\nu_i^2=||\mathbf{P}\bmu||^2=||\bmu||^2$$ give the bound
\begin{align*}
\notag& \sum_{i=1}^n \Big\{\frac{\nu_i^2}{2(1-t\lambda_i)}-\log\cosh(\mu_i)\Big\}-\sum_{i=1}^n\frac{1}{2}\log(1-t\lambda_i)\\
 \le &\Big\{\frac{1}{2}\sumn \nu_i^2+\frac{t}{2}\sum_{i=1}^n\nu_i^2\lambda_i+t^2\gamma_\rho\sum_{i=1}^n\nu_i^2 \lambda_i^2\Big\}+\Big\{-\frac{1}{2}\sum_{i=1}^n\mu_i^2+\sum_{i=1}^n\mu_i^4\Big\}+\gamma_\rho t^2\sum_{i=1}^n\lambda_i^2\\
=&\frac{t}{2}\bmu^\top\mathbf{Q}\bmu+t^2\gamma_\rho \bmu^\top\mathbf{Q}^2\bmu+\sum_{i=1}^n\mu_i^4+\gamma_\rho t^2 \sum_{i,j=1}^n\bQ_{ij}^2\\
\le &\frac{t}{2}C\rho \sqrt{n}+t^2\gamma_\rho C\rho^2\sqrt{n}+C\sqrt{n}+\gamma_\rho t^2D\sqrt{n},
\end{align*}
where $D>0$ is such that $\sum_{i,j=1}^n\bQ_{ij}^2\le D\sqrt{n}$.
This along with \eqref{eqn:bound} gives
\begin{align*}
\Big[\frac{1}{2}C(1+t\rho) +t^2\gamma_\rho C\rho^2+C+\gamma_\rho t^2D\Big]\sqrt{n}\ge &\frac{1}{2}\E_{\mathbf{Q},\bmu}\bX^\top\mathbf{Q}\bX=\frac{1}{2}\E_{\mathbf{Q},\bmu}\sum_{i=1}^nX_im_i(\bX)
\end{align*}
But, for some random $(\xi_i, i=1,\ldots,n)$
\begin{align*}
\ & \frac{1}{2}\E_{\mathbf{Q},\bmu}\sum_{i=1}^nX_im_i(\bX)\\
&=\frac{1}{2}\E_{\mathbf{Q},\bmu}\sum_{i=1}^n\tanh(m_i(\bX)+\mu_i)m_i(\bX)\\
&=\frac{1}{2}\E_{\mathbf{Q},\bmu}\sum_{i=1}^n\tanh(m_i(\bX))m_i(\bX)+\frac{1}{2}\E_{\mathbf{Q},\bmu}\sum_{i=1}^n \mu_i m_i(\bX)\sech^2(\xi_i).
\end{align*}
Now, 
\begin{align*}
\frac{1}{2}\E_{\mathbf{Q},\bmu}\sum_{i=1}^n\tanh(m_i(\bX))m_i(\bX)\ge &\frac{\eta}{2} \E_{\mathbf{Q},\bmu}\sum_{i=1}^nm_i(\bX)^2,
\end{align*}
where
$$\eta:=\inf_{|x|\le 1}\frac{\tanh(x)}{x}>0.$$
The desired conclusion of the lemma follows by noting that
\begin{align*}
\Big|\E_{\mathbf{Q},\bmu}\sum_{i=1}^n \mu_i m_i(\bX)\sech^2(\xi_i)\Big|\leq C\sqrt{n}.
\end{align*}
\end{proof}

\begin{proof}[Proof of Lemma \ref{lem:bayesian}]
We begin with Part (a). By a simple algebra,  the p.m.f. of $\bX$ can  be written as
$$\P_{\theta,\bmu}(\bX=\bx)\propto \text{exp}\left\{\frac{n\theta}{2}\bar{x}^2+\sum_{i=1}^nx_i\mu_i\right\}.$$
Consequently, the joint density of $(\bX,Z_n)$  with respect to the product measure of counting measure on $\{-1,1\}^n$ and Lebesgue measure on $\R$ is proportional to
\begin{align*}
&\text{exp}\left\{\frac{n\theta}{2}\bar{x}^2+\sum_{i=1}^n x_i\mu_i-\frac{n\theta}{2}(z-\bar{x})^2\right\}\\
=&\text{exp}\left\{-\frac{n\theta}{2}z^2+\sum_{i=1}^nx_i(\mu_i+z\theta)\right\}.
\end{align*}
Part (a) follows from the expression above.

Now consider Part (b). Using the joint density of Part (a), the marginal density of $Z_n$ is proportional to
\begin{align*}
&\sum_{{\bf x}\in \{-1,1\}^n}\text{exp}\left\{-\frac{n\theta}{2}z^2+\sum_{i=1}^nx_i(\mu_i+z\theta)\right\}\\
=&\text{exp}\left\{-\frac{n\theta}{2}z^2+\sum_{i=1}^n\log\cosh(\mu_i+z\theta)\right\}=e^{-f_{n,\bmu}(z)},
\end{align*}
thus completing the proof of Part (b).

Finally, consider Part (c). By Part (a)  given $Z_n=z$ the random variables 
$(X_1,\cdots,X_n)$ are independent, with $$\Ptheta(X_i=1|Z_n=z)=\frac{e^{\mu_i+\theta z}}{e^{\mu_i+\theta z}+e^{-\mu_i-\theta z}},$$
and so $$\Etheta(X_i|Z_n=z)=\tanh(\mu_i+\theta z),\quad  {\sf Var}_{\theta,\bmu}(X_i|Z_n=n)=\text{sech}^2(\mu_i+\theta z).$$
Thus for any $\bmu\in [0,\infty)^n$ we have
\begin{align*}
\Etheta\Big(\sum_{i=1}^n (X_i-\tanh(\mu_i+\theta Z_n))\Big)^2=& \Etheta\Etheta\Big\{\Big(\sum_{i=1}^n (X_i-\tanh(\mu_i+\theta Z_n))\Big)^2\Big|Z_n\Big\}\\
=&\E\sum_{i=1}^n\text{sech}^2(\mu_i+\theta Z_n)\le n.
\end{align*}
\end{proof}

\begin{proof}[Proof of Lemma \ref{lem:theta1}]
We begin with Part (a). Since
$$f_{n,\bmu}''(z)=\sum_{i=1}^n \text{tanh}^2( z+\mu_i)$$
is strictly positive for all but at most one $z\in \R$,
the function $z\mapsto f_{n,\bmu}(z)$ is strictly convex with $f_{n,\bmu}(\pm \infty)=\infty$, it follows that $z\mapsto f_{n,\bmu}(z)$ has a unique minima $m_n$ which is the unique root of the equation $f_{n,\bmu}'(z)=0$.  The fact that $m_n$ is positive follows on noting that
$$f_{n,\bmu}'(0)=-\sum_{i=1}^n \tanh(\mu_i)<0, \quad f_{n,\bmu}'(+\infty)=\infty.$$
Also $f_n'(m_n)=0$ gives
$$m_n=\frac{1}{n}\sum_{i=1}^n\tanh(m_n+\mu_i)\le1,$$
and so $m_n\in (0,1]$. Finally, $f_{n,\bmu}'(m_n)=0$ can be written as
\begin{align*}
m_n-\tanh(m_n)=\frac{s}{n}\Big[\tanh( m_n+B)-\tanh( m_n)\Big]\ge C\frac{s}{n}\tanh(B), 
\end{align*}
for some $C>0$, which proves Part (a).

Now consider Part (b). By a Taylor's series expansion around $m_n$ and using the fact that $f_n''(z)$ is strictly increasing on $(0,\infty)$ gives
\begin{align*}
&f_n(z)\ge f_n(m_n)+\frac{1}{2}(z-m_n)^2f_n''(m_n+Kn^{-1/4})\text{ for all }z\in[m_n+K n^{-1/4},\infty)\\
&f_n(z)\le f_n(m_n)+\frac{1}{2}(z-m_n)^2f_n''(m_n+Kn^{-1/4})\text{ for all }z\in [m_n,m_n+K n^{-1/4}].
\end{align*}
 Setting $b_n:= f_n''(m_n+Kn^{-1/4})$ this gives
\begin{align*}
&\Ptheta(Z_n>m_n+K n^{-1/4})\\
=&\frac{\int_{m_n+Kn^{-1/4}}e^{- f_n(z)}dz}{\int_\R e^{- f_n(z)}dz}\\
\le & \frac{\int_{m_n+K n^{-1/4}}^\infty e^{-\frac{b_n}{2}(z-m_n)^2}dz}{\int_{m_n}^{m_n+Kn^{-1/4}} e^{-\frac{b_n}{2}(z-m_n)^2}dz}\\
=&\frac{\P(N(0,1)>K n^{-1/4}\sqrt{b_n})}{\P(0<N(0,1)<K n^{-1/4}\sqrt{b_n})},
\end{align*}
from which the desired conclusion will follow if we can show that $\liminf_{n\rightarrow\infty}n^{-1/2}b_n>0$. But this follows on noting that $$n^{-1/2}b_n=n^{-1/2}f_n''(m_n+Kn^{-1/4}))\ge \sqrt{n}\tanh^2(K n^{-1/4})=K^2\Theta(1).$$

Finally, let us prove Part (c). By a Taylor's series expansion about $\delta m_n$ and using the fact that $f_n(\cdot)$ is convex with unique global minima at $m_n$ we have
$$f_n(z) \ge f_n(m_n)+(z-\delta m_n)f_n'(\delta m_n),\quad \forall z\in(-\infty, \delta m_n].$$
Also, as before we have
$$f_n(z)\le f_n(m_n)+\frac{1}{2}(z-m_n)^2f_n''(m_n),\forall z\in [m_n,2m_n]$$
Thus with $c_n:=f_n''(2m_n)$ for any $\delta>0$ we have
\begin{align}
\notag\Ptheta(Z_n\le \delta m_n)=&\frac{\int_{-\infty}^{\delta m_n}e^{-f_n(z)}dz}{\int_{\R}e^{-f_n(z)}dz}\\
\notag\le &\frac{\int_{-\infty}^{\delta m_n} e^{-(z-\delta m_n)f_n'(\delta m_n)}dz}{\int_{m_n}^{2m_n} e^{-\frac{c_n}{2}(z-m_n)^2}dz}\\
=&\frac{\sqrt{2\pi  c_n}}{|f_n'(\delta m_n)|\P(0<Z<m_n\sqrt{c_n})}.
\label{eq:theta=1}
\end{align}
To bound the the rightmost hand side of \eqref{eq:theta=1}, we claim that the following estimates hold:
\begin{align}
 c_n=&\Theta(nm_n^2)\label{eq:theta=1a},\\
 nm_n^3=&O(|f_n'(\delta m_n)|)\label{eq:theta=1b}.
 \end{align}
 Given these two estimates, we immediately have
 \begin{align}\label{eq:theta=1c}
 m_n\sqrt{c_n}=\Theta(m_n^2\sqrt{n})\ge \Theta(A_n^{2/3}\sqrt{n})\rightarrow \infty,
 \end{align}
 as $A_n\gg n^{-3/4}$ by assumption.
 Thus the rightmost hand side of \eqref{eq:theta=1} can be bounded by
 \begin{align*}
 \frac{m_n\sqrt{n}}{n m_n^3}=\frac{1}{m_n^2\sqrt{n}}\rightarrow 0,
 \end{align*} 
 where the last conclusion uses \eqref{eq:theta=1c}. This completes the proof of Part (c). 
 
 It thus remains to prove the estimates \eqref{eq:theta=1a} and \eqref{eq:theta=1b}. To this effect, note that
 \begin{align*}
 f_n''(2m_n)=&\sum_{i=1}^n\tanh^2(2m_n+\mu_i)\\
\le &\sum_{i=1}^n\Big(\tanh(2m_n)+C_1\tanh(\mu_i)\Big)^2\\
\le & 2n \tanh^2(2m_n)+2C_1^2\sum_{i=1}^n\tanh^2(\mu_i)\\
\lesssim &nm_n^2+nA(\mu_n)\lesssim nm_n^2,
\end{align*}
where the last step uses part (a), and $C_1<\infty$ is a universal constant. This gives the upper bound in \eqref{eq:theta=1a}. For the lower bound of \eqref{eq:theta=1a} we have
\begin{align*}
f_n''(m_n)=\sum_{i=1}^n\tanh^2(2m_n+\mu_i)
\ge n\tanh^2(2m_n)\gtrsim nm_n^3.
\end{align*}
Turning to prove \eqref{eq:theta=1b} we have
\begin{align*}
|f_n'(\delta m_n)|=&\sum_{i=1}^n\tanh(\delta m_n+\mu_i)-n\delta m_n\\
=&\Big[\sum_{i=1}^n\tanh(\delta m_n+\mu_i)-\tanh(\delta m_n)\Big]-n[\delta m_n-\tanh(\delta m_n)]\\
\ge &C_2 nA(\mu_n)-C_3n\delta^3m_n^3\\
\gtrsim &n m_n^3,
\end{align*}
where $\delta$ is chosen small enough, and $C_2>0,C_3<\infty$ are universal constants. This completes the proof of \eqref{eq:theta=1b}, and hence completes the proof of the lemma.
\end{proof}

\end{document}